\documentclass{amsart}

\usepackage[utf8]{inputenc}
\usepackage{amsmath}
\usepackage{amssymb}
\usepackage{mathdots}
\usepackage{mathtools}
\usepackage{mathrsfs}
\usepackage{tikz}
\usepackage{tikz-cd}
\usetikzlibrary{decorations.markings,arrows.meta,quotes}
\usepackage[inline]{enumitem}
\usepackage{color}
\usepackage{booktabs}
\definecolor{linkcolor}{rgb}{0,0,0.6}
\usepackage{amsthm}
\usepackage{bbm}
\usepackage{multicol}

\usepackage[english]{babel}
\usepackage{csquotes}

\usepackage[colorlinks=true,
pdfstartview=Fit,
linkcolor= linkcolor,
citecolor= linkcolor,
urlcolor= linkcolor,
hyperindex=true,
hyperfigures=false]
{hyperref}

\usepackage{orcidlink}

\tikzset{
  goodcycle/.style={draw=blue!65!black,fill=blue!8,
    rounded corners=2pt,inner sep=4pt},
  badcycle/.style={draw=red!70!black,fill=red!8,
    rounded corners=2pt,inner sep=4pt}
}

\hypersetup{pdftitle={Unipotent quotients in compactified Coxeter varieties},
pdfauthor={Joseph Muller}}

\newtheorem{theorem}{Theorem}[section]
\newtheorem{proposition}[theorem]{Proposition}
\newtheorem{lemma}[theorem]{Lemma}
\newtheorem{corollary}[theorem]{Corollary}
\theoremstyle{definition}
\newtheorem{definition}[theorem]{Definition}
\theoremstyle{remark}
\newtheorem{remark}[theorem]{Remark}
\numberwithin{equation}{section}

\title[Unipotent quotients in Coxeter varieties]{Unipotent quotients in compactified Coxeter varieties}
\author{Joseph Muller\,\orcidlink{0000-0002-1546-0910}}\thanks{National Center for Theoretic Sciences, National Taiwan University, No. 1, Sec. 4, Roosevelt Rd., Taipei, Taiwan, muller@ncts.ntu.edu.tw}
\date{}

\begin{document}


\begin{abstract}
In this paper, we extend Lusztig’s description of quotients of Coxeter varieties by unipotent radicals of parabolic subgroups, to quotients of certain open neighbourhoods in the smooth compactification of Coxeter varieties. These quotients are products of an affine space with a Coxeter variety for an associated Levi quotient, and this result holds uniformly for any connected reductive group over a finite field. As an application, we prove a conjecture of Wang computing the invariants, under the action of unipotent subgroups, of the cohomology of the Deligne-Lusztig coverings of Coxeter varieties. We also prove that derived Alvis–Curtis duality sends the compactly supported cohomology of these coverings to their ordinary cohomology. 
\end{abstract}

\maketitle

\tableofcontents

\section*{Introduction}

In \cite{deligneRepresentationsReductiveGroups1976a}, Deligne and Lusztig constructed representations of finite reductive groups using the $\ell$-adic cohomology of certain algebraic varieties associated to elements of the Weyl group. The varieties associated to Coxeter elements were studied by Lusztig in \cite{lusztigCoxeterOrbitsEigenspaces1976a}, who determined their $\ell$-adic cohomology and the action of the Frobenius. The cohomology of these varieties and of their Galois coverings with integral and modular coefficients was subsequently studied by Bonnafé and Rouquier in \cite{bonnafe_coxeter_2006}, and Dudas in \cite{dudas_coxeter_2014-1}.

Let $G$ be a connected reductive group over $\overline{\mathbb F_q}$, where $q$ is a power of a prime $p$, and let $F$ be a Frobenius morphism defining an $\mathbb F_q$-structure on $G$. Fix a Coxeter element $c$ for $(G, F)$ and write $X = X_G(c)$ for the associated Coxeter variety. A subgroup $H$ of $G$ is called rational if it is $F$-stable, and in this case $H^F$ denotes the subgroup of $F$-fixed points. For a rational parabolic $P$, denote its unipotent radical by $U_P$ and its Levi quotient by $L_P = P/U_P$. The finite group $G^F$ acts on $X$, and the quotient by $U_P^F \subset G^F$ admits a product decomposition $X/U_P^F \simeq S_P \times \mathbb G_m^{d(P)}$, where $S_P$ is a Coxeter variety for $L_P$ with its induced Frobenius morphism, and $d(P)$ is the relative corank of $P$. This decomposition is due to Lusztig, see \cite{lusztigCoxeterOrbitsEigenspaces1976a} Theorem 2.6 and the proof of Corollary 2.10. We recall its construction in Section \ref{subsec:lusztig}.

Our main result extends these decompositions across the boundary of the Bott-Samelson-Demazure-Hansen compactification $\overline X$ of $X$. This compactification is smooth and projective, and $\overline X \setminus X$ is a divisor with simple normal crossings, see for instance \cite{deligneRepresentationsReductiveGroups1976a}. Its strata $C_P$ are indexed by the rational parabolic subgroups, and the stratum $C_G = X$ is open. For each $P$, the \textit{open star} of $C_P$ is the union $C_P^* = \bigsqcup_{Q \supseteq P} C_Q$, which is an open neighbourhood of $C_P$. Theorem \ref{thm:star-quotient} shows that, after possibly inverting some coordinates on $\mathbb G_m^{d(P)}$, the preceding decomposition extends to an isomorphism
\begin{equation}\label{eq:intro-main-theo}
    C_P^*/U_P^F \xrightarrow{\sim} S_P \times \mathbb A^{d(P)}. \tag{$*$}
\end{equation}
The reduced images of the boundary divisors of $C_P^*$ correspond to the products of $S_P$ with the coordinate hyperplanes in $\mathbb A^{d(P)}$, and the reduced image of $C_P$ is $S_P \times \{0\}$. In particular, $C_P^*$ is affine and the quotient morphism $C_P^* \to C_P^*/U_P^F$ is finite flat, see Corollary \ref{cor:star-affine-flat}.

The proof, which is the subject of Section \ref{sec:geometry}, first treats the case where $P$ is a maximal proper rational parabolic, in which case $C_P^* = X \sqcup C_P$. We extend the projection $X \to S_P$ to $X \sqcup C_P$ and prove that its restriction to $C_P$ is dominant. The latter assertion follows from the fact that any $G^F$-equivariant rational map between two Coxeter varieties for the same pair $(G, F)$ is automatically dominant, see Proposition \ref{prop:dominance}. The proof concludes with a point counting argument, and the case of arbitrary rational parabolic reduces to the maximal case by taking successive quotients. 

The first application of \eqref{eq:intro-main-theo} is a proof of Wang's conjecture stated in \cite{wangCohomologieCompactificationVarietes2014}. Let $\pi: Y = Y(\dot c) \to X$ be the Deligne-Lusztig covering of $X$ associated with a representative $\dot c$ of $c$. It is a torsor under $T^{cF}$, namely the group of fixed points of the map $t \mapsto \dot cF(t)\dot c^{-1}$ defined on the rational maximal torus $T$ used to construct $c$. Fix $\Lambda = \mathbb Z/\ell^a\mathbb Z$, with $\ell \not = p$ a prime and $a \geq 1$. Let us write $j: X \hookrightarrow \overline X$ and $i_P: \overline{C_P} \hookrightarrow \overline X$ for the natural immersions. Wang's Conjecture A asserts that the restriction morphism on cohomology induces an isomorphism
\begin{equation}\label{eq:intro-wang-conj}
    R\Gamma(X, \pi_*\Lambda)^{U_P^F} \xrightarrow{\sim} R\Gamma(\overline{C_P}, i_P^*Rj_*\pi_*\Lambda), \tag{$**$}
\end{equation}
for every proper rational parabolic $P$. Wang proved this assertion for $G = \mathrm{GL}_d$ with its standard Frobenius and the Coxeter element $(1\ \ldots\ d) \in \mathfrak S_d$. His proof relies crucially on the isomorphism \eqref{eq:intro-main-theo} for maximal proper rational parabolics, which he proved using the explicit description of $\overline X$ as a sequence of blow-ups. Our main theorem \eqref{eq:intro-main-theo} supplies this geometric input uniformly for every $(G, F)$, and therefore allows us to adapt Wang's proof of \eqref{eq:intro-wang-conj} in this generality, see Theorem \ref{thm:wang}. This is the subject of Section \ref{sec:restriction}. In the case of $\mathrm{GL}_d$, Wang used \eqref{eq:intro-wang-conj} to compute the cohomology of the tamely ramified covering of Drinfeld's $p$-adic symmetric space and relate it to the local Langlands and to the Jacquet-Langlands correspondences in \cite{wangLespaceSymetriqueDrinfeld2014a, wangLespaceSymetriqueDrinfeld2017a}.

We provide another application of \eqref{eq:intro-wang-conj} concerning Alvis-Curtis duality. In Corollary \ref{cor:coxeter-duality}, we establish an isomorphism
\begin{equation}\label{eq:intro-alvis-curtis}
    D_G R\Gamma_c(Y, \Lambda) \simeq R\Gamma(Y, \Lambda), \tag{$***$}
\end{equation}
where $D_G$ is the derived Alvis-Curtis functor of \cite{cabanes_alvis-curtis_2001, cabanes_okuyamas_2009} defined with the conventions specified in Section \ref{sec:alvis-curtis}. This isomorphism respects the commuting actions of $G^F$ and $T^{cF}$. The proof uses the \v{C}ech complex for the covering of $\overline X \setminus X$ by the boundary divisors. The isomorphism \eqref{eq:intro-wang-conj} then identifies the resulting terms with $U_P^F$-invariants, which in turn identifies the resulting complex with the one defining the inverse of the Alvis-Curtis functor. Both \eqref{eq:intro-wang-conj} and \eqref{eq:intro-alvis-curtis} also hold with $\mathbb Z_{\ell}$ and $\mathbb Q_{\ell}$-coefficients.

\section{Coxeter varieties and unipotent quotients}\label{sec:notation}

Throughout, $q$ denotes a power of a prime $p$. Let $G$ be a connected reductive group over $\overline{\mathbb F_q}$, endowed with a Frobenius endomorphism $F$ defining an $\mathbb{F}_q$-structure. A subgroup is called \textit{rational} if it is $F$-stable. For a rational parabolic subgroup $P$, write $U_P$ for its unipotent radical and $L_P = P/U_P$ for its Levi quotient. 

\subsection{Coxeter varieties and their compactification}

Fix a rational Borel subgroup $B$ and a rational maximal torus $T \subset B$, and put $U = R_u(B)$. Let $W \coloneq N_G(T)/T$ be the Weyl group, and write $\Delta$ for the set of simple reflections determined by $B$. The action of $F$ on $W$ permutes $\Delta$ and we denote this permutation by $\phi$. Put $r \coloneq |\Delta/\phi|$. Choose representatives $s_1, \ldots, s_r$ of the $\phi$-orbits in $\Delta$, and set
\begin{equation*}
    c \coloneq s_1 \cdots s_r.
\end{equation*}
Thus $c$ is a \textit{Coxeter element} for $(G, F)$. Throughout the paper and unless mentioned otherwise, the element $c$ and its given reduced expression are fixed. 

Let $\mathcal B(G)$ be the variety of Borel subgroups of $G$. For $B_1, B_2 \in \mathcal B(G)$, we write $\mathrm{pos}(B_1, B_2) \in W$ for their relative position. The Deligne-Lusztig variety associated with $w \in W$ is
\begin{equation*}
    X_G(w) \coloneq \{B' \in \mathcal B(G) \mid \mathrm{pos}(B', F(B')) = w\}.
\end{equation*}
    The group $G^F$ acts on $X_G(w)$ by conjugation. We write $X \coloneq X_G(c)$, and refer to $X$ as a \textit{Coxeter variety} for $(G, F)$. Its Bott-Samelson-Demazure-Hansen compactification is
\begin{equation}\label{eq:compactification}
     \overline{X} \coloneq \left\{ (B_0, \ldots, B_r) \in \mathcal B(G)^{r+1} \ \middle|\ 
    \begin{array}{ll}
        \forall 1 \leq j \leq r, & \mathrm{pos}(B_{j-1}, B_j) \leq s_j,\\
        & B_r = F(B_0)
    \end{array} \right\}.
\end{equation}
Since the simple reflections $s_i$ are pairwise distinct, the first projection $(B_0, \ldots, B_r) \mapsto B_0$ identifies $\overline{X}$ with the Zariski closure of $X$ in $\mathcal B(G)$. Furthermore, the open subset where every relative position is $s_j$ is identified with $X$. The variety $\overline{X}$ is smooth and projective of dimension $r$, and $\overline{X} \setminus X$ is a divisor with simple normal crossings, see for instance Paragraph 9 of \cite{deligneRepresentationsReductiveGroups1976a} and Section 2.2 of \cite{wangCohomologieCompactificationVarietes2014}.

\subsection{Boundary strata and open stars}\label{subsec:strata}

For a $\phi$-stable subset $I \subseteq \Delta$, let $P_I$ be the standard rational parabolic subgroup of type $I$. Write $U_I = U_{P_I}$ and let $L_I$ be the unique Levi subgroup of $P_I$ containing $T$. The quotient map $P_I \to L_{P_I}$ restricts to an isomorphism $L_I \xrightarrow{\sim} L_{P_I}$ which we use throughout. Let $c_I$ be the subword of $c$ obtained by retaining only the letters $s_j \in I$ in their specific order. The locus in \eqref{eq:compactification} on which
\begin{equation*}
    \mathrm{pos}(B_{j-1}, B_j) = \begin{cases} 
        s_j & \text{if } s_j \in I,\\
        1 & \text{if } s_j \not \in I,
    \end{cases}
\end{equation*}
is identified with $X_G(c_I)$ via the first projection. Its connected components are indexed by $G^F/P_I^F$. The component indexed by the identity coset consists of the points with $B_0 \subseteq P_I$, and the map $B_0 \mapsto B_0 \cap L_I$ identifies it with $X_{L_I}(c_I)$, which is a Coxeter variety for $(L_I, F_{|L_I})$.

Every rational parabolic $P$ is conjugate under $G^F$ to a unique $P_I$, and we say that $P$ is a parabolic \textit{of type $I$}. We write $C_P \subset \overline{X}$ for the corresponding component of $X_G(c_I)$, and put $C_G = X$. We define
\begin{equation*}
    r(P) \coloneq |I/\phi|, \qquad d(P) \coloneq r-r(P).
\end{equation*}
In particular, $d(P)$ is the relative corank of $P$. We recall from \cite{wangCohomologieCompactificationVarietes2014} Section 2.2 that
\begin{equation}\label{eq:stratification}
    \overline{X} = \bigsqcup_P C_P, \qquad  \overline{C_P} = \bigsqcup_{Q \subseteq P} C_Q,
\end{equation}
where the unions are indexed by rational parabolic subgroups, with $G$ being allowed. The map $B_0 \mapsto B_0/U_P$ identifies $C_P$ with a Coxeter variety for $(L_P, F_P)$, where $F_P$ is the Frobenius morphism on $L_P$ induced by $F$. It also identifies $\overline{C_P}$ with the Bott-Samelson-Demazure-Hansen compactification of this Coxeter variety. In particular, $C_P$ is smooth, affine and irreducible of dimension $r(P)$. The boundary divisors of $\overline{X}$ are the $\overline{C_M}$ for $M$ maximal among proper rational parabolics. For any fixed type, these divisors are pairwise disjoint and form a single $G^F$-orbit. Indeed, the first Borel $B_0$ of a point in $\overline{C_M}$ is contained in $M$, and a Borel is contained in a unique parabolic of a given type. 

The fixed-point description in Section 2.2 of \cite{wangCohomologieCompactificationVarietes2014} gives
\begin{equation}\label{eq:fixed-locus}
    \bigl( \overline{X}^{U_P^F} \bigr)_{\mathrm{red}} = \overline{C_P}.
\end{equation}
In particular, $U_P^F$ acts trivially on $\overline{C_P}$, and the action of $P^F$ on this variety factors through $P^F/U_P^F \simeq L_P^{F_P}$. 

\begin{definition}\label{def:star}
    For a rational parabolic $P$, the \textit{open star of the stratum $C_P$} is the open subvariety 
    \begin{equation*}
        C_P^* \coloneq \bigsqcup_{Q \supseteq P} C_Q \subseteq \overline{X}.
    \end{equation*}
\end{definition}

The open star $C_P^*$ is a $P^F$-stable open subvariety of $\overline{X}$. Its complement is the union of the divisors $\overline{C_M}$ for which $M$ does not contain $P$. There are exactly $d(P)$ maximal proper rational parabolics containing $P$. Their corresponding divisors, intersected with $C_P^*$, are the irreducible components of $C_P^* \setminus X$. The intersection of these irreducible components is $C_P$. In particular, if $M$ is a maximal proper rational parabolic then $C_M^* = X \sqcup C_M$.

\subsection{Lusztig's quotient construction} \label{subsec:lusztig}

We recall the description of unipotent quotients of $X$ following the arguments of \cite{lusztigCoxeterOrbitsEigenspaces1976a} Theorem 2.6 and Corollary 2.10. We explain the constructions in detail as they will be used subsequently, but we claim no novelty for the contents of this section. 

Let $B^-$ be the Borel opposite to $B$ with respect to $T$. Let $w_0$ be the longest element of $W$. We denote by $\tau_G$ the involution induced by $w_0$ on simple reflections and on simple roots, namely
\begin{equation*}
    \tau_G(s) \coloneq w_0sw_0, \qquad  \tau_G(\alpha) \coloneq -w_0(\alpha).
\end{equation*}
We write $\alpha_j$ for the simple root corresponding to $s_j$, and we fix root parametrizations $x_{\alpha}: \mathbb G_a \xrightarrow{\sim} U_{\alpha}$ for the simple roots of $G$. We define a map
\begin{equation*}
    A: \mathbb{G}_m^r \to U, \qquad (t_1, \ldots, t_r) \mapsto x_{\tau_G(\alpha_1)}(t_1) \cdots x_{\tau_G(\alpha_r)}(t_r).
\end{equation*}
We write $\mathcal L_H(h) = h^{-1}F(h)$ for the Lang map of a connected rational subgroup $H$. By \cite{lusztigCoxeterOrbitsEigenspaces1976a} Theorem 2.6, there is a $U^F$-equivariant isomorphism
\begin{equation}\label{eq:lang-locus}
    \begin{split}
        Z_c \coloneq \{u \in U \mid \mathcal L_U(u) = A(t) \text{ for some } t \in \mathbb{G}_m^r\} & \xrightarrow{\sim} X,\\
        u & \mapsto uB^- u^{-1}.
    \end{split}
\end{equation}
For $u \in Z_c$, the parameters $t_j$ are uniquely determined and we denote them by $a_j(u)$. This defines regular morphisms $a_j:Z_c \to \mathbb{G}_m$, which we also regard as functions on $X$ through \eqref{eq:lang-locus}. The action of $U^F$ on $Z_c$ is left multiplication, so it is free. Moreover, for every $h \in U^F$ we have $\mathcal L_U(hu) = \mathcal L_U(u)$, and hence every $a_j$ is $U^F$-invariant. Let $I \subseteq \Delta$ be $\phi$-stable, and define $V_I \coloneq U \cap L_I$. 

Thus $V_I$ is the unipotent radical of the Borel $B \cap L_I \in \mathcal B(L_I)$. Upon intersecting with $U$, the Levi decomposition of $P_I$ yields the decomposition $U = U_I \rtimes V_I$. Let $\pi_I: U \to V_I$ be the corresponding projection, which commutes with $F$. If $w_I$ is the longest element of $W_I$, we denote by $\tau_I$ the corresponding involution acting on the simple reflections and simple roots of $L_I$ via
\begin{equation*}
    \tau_I(s) \coloneq w_Isw_I, \qquad  \tau_I(\alpha) \coloneq -w_I(\alpha).
\end{equation*}
Finally, we also define
\begin{equation}\label{eq:levi-word}
    c'_I \coloneq \prod_{\substack{j \text{ such that} \\ \tau_G(s_j)\in I}} \tau_I \tau_G(s_j), \qquad S_{P_I} \coloneq X_{L_I}(c'_I).
\end{equation}
The product defining $c_I'$ is taken in increasing order of $j$. Both $\tau_I$ and $\tau_G$ commute with $\phi$, thus $c'_I$ is a Coxeter element for $(L_I, F_{|L_I})$, and $S_{P_I}$ is a Coxeter variety thereof. We note that in general, $c_I \not = c_I'$, therefore $C_{P_I}$ and $S_{P_I}$ are Coxeter varieties associated to possibly different Coxeter elements. The indices omitted in the product defining $c_I'$ are
\begin{equation*}
    J(P_I) \coloneq \{j \in \{1, \ldots, r\} \mid \tau_G(s_j) \not \in I\}, \qquad |J(P_I)|= d(P_I).
\end{equation*}
The proof of the following proposition is adapted from \cite{lusztigCoxeterOrbitsEigenspaces1976a} Corollary 2.10.

\begin{proposition}\label{prop:open-quotient}
    The morphism
    \begin{equation*}
        \rho_{P_I}: X \to S_{P_I}, \qquad B' \mapsto (B' \cap P_I) U_I/U_I,
    \end{equation*}
    defined using the identification $P_I/U_I \simeq L_I$ is $P_I^F$-equivariant, with the action on $S_{P_I}$ factoring through $L_I^F$. The morphism
    \begin{equation*}
        \left(\rho_{P_I}, (a_j)_{j \in J(P_I)} \right): X \to S_{P_I} \times \mathbb G_m^{J(P_I)}
    \end{equation*}
    induces an isomorphism
    \begin{equation*}
        \theta_I: X/U_I^F \xrightarrow{\sim} S_{P_I} \times \mathbb G_m^{J(P_I)}.
    \end{equation*}
\end{proposition}

\begin{proof}
The projection $\pi_I$ restricts to the identity on $U_{\tau_G(\alpha_j)}$ when $\tau_G(s_j) \in I$, and kills this root subgroup otherwise. Consequently, $\pi_I(A(t)) = \prod_{\tau_G(s_j) \in I} x_{\tau_G(\alpha_j)}(t_j)$. We define
\begin{equation}\label{eq:retained-lang}
    Z_I \coloneq \left\{v \in V_I \ \middle| \  v^{-1}F(v) = \prod_{\substack{j \text{ such that} \\ \tau_G(s_j) \in I}} x_{\tau_G(\alpha_j)}(t_j) \text{ for some } t_j \in \mathbb G_m \right\}, 
\end{equation}
where the product is taken in increasing order of $j$. Since $\pi_I$ commutes with $F$, it maps $Z_c$ into $Z_I$. Applying \eqref{eq:lang-locus} with respect to the Coxeter variety $S_{P_I}$ for $(L_I, F_{|L_I})$ gives an isomorphism
\begin{equation*}
    Z_I \xrightarrow{\sim} S_{P_I}, \qquad v \mapsto v(B^-\cap L_I)v^{-1}.
\end{equation*}

Consider
\begin{equation*}
    E_I \coloneq \{(v, t) \in V_I \times \mathbb G_m^r \mid \mathcal L_{V_I}(v) = \pi_I(A(t))\}.
\end{equation*}
The parameters $t_j$ with $\tau_G(s_j)\in I$ are uniquely determined by $v$ and depend regularly on it. Hence
\begin{equation*}
    E_I \xrightarrow{\sim} Z_I \times \mathbb G_m^{J(P_I)}, \qquad (v,t) \mapsto \left(v, (t_j)_{j \in J(P_I)} \right).
\end{equation*}
For $(v, t) \in E_I$, the element $vA(t)F(v)^{-1}$ belongs to $\mathrm{Ker}(\pi_I) = U_I$. Writing $u = nv$ with $n \in U_I$, we obtain
\begin{equation*}
    \mathcal L_U(u) = A(t) \quad \iff \quad \mathcal L_{U_I}(n) = vA(t)F(v)^{-1}.
\end{equation*}
Consequently, we have an isomorphism
\begin{equation*}
    E_I \times_{U_I} U_I \xrightarrow{\sim} Z_c, \qquad ((v, t), n) \mapsto nv,
\end{equation*}
where the fiber product is formed using $(v,t) \mapsto vA(t)F(v)^{-1}$ on the first factor and $\mathcal L_{U_I}$ on the second. Under this isomorphism, the projection to $E_I$ is
\begin{equation*}
    Z_c \to E_I, \qquad u \mapsto \left( \pi_I(u), (a_1(u), \ldots, a_r(u)) \right).
\end{equation*}
Since the Lang map $\mathcal L_{U_I}$ is a finite étale $U_I^F$-torsor, so is this projection. We therefore obtain an isomorphism $Z_c/U_I^F \xrightarrow{\sim} E_I$. Together with \eqref{eq:lang-locus} and the identifications above, we finally constructed an isomorphism
\begin{equation*}
    \theta_I: X/U_I^F \xrightarrow{\sim} S_{P_I} \times \mathbb G_m^{J(P_I)}.
\end{equation*}
It remains to identify the morphism on $X$ inducing $\theta_I$. Let $B' = uB^-u^{-1}$ with $u \in Z_c$, and write $u = nv$, where $n \in U_I$ and $v = \pi_I(u)$. Since $u \in P_I$ and $B^-\cap P_I = B^- \cap L_I$, we have
\begin{equation*}
    B'\cap P_I = nv(B^-\cap L_I)v^{-1}n^{-1}.
\end{equation*}
Under the identification $P_I/U_I \simeq L_I$, its image in the quotient by $U_I$ is therefore $\rho_{P_I}(B') = v(B^-\cap L_I)v^{-1}$. Consequently, the composite
\begin{equation*}
    X\to X/U_I^F \xrightarrow{\theta_I} S_{P_I} \times \mathbb G_m^{J(P_I)}
\end{equation*}
is precisely $\left(\rho_{P_I}, (a_j)_{j\in J(P_I)}\right)$. Finally, intersection with $P_I$ and passage to $P_I/U_I$ are compatible with conjugation by $P_I$. Thus $\rho_{P_I}$ is $P_I^F$-equivariant, where $P_I^F$ acts on $S_{P_I}$ through the quotient $P_I^F/U_I^F\simeq L_I^F$.
\end{proof}

While the morphism $\rho_{P_I}$ is $P_I^F$-equivariant, the functions $a_j: X \to \mathbb G_m$ for $j \in J(P_I)$ need not be so. Consequently, the isomorphism $\theta_I$ need not be $L_I^F$-equivariant when $(\mathbb G_m)^{J(P_I)}$ is given the trivial $L_I^F$-action. We next describe successive quotients for embedded standard parabolics.

\begin{lemma}\label{lem:nested-quotients}
    Let $I \subseteq K \subseteq \Delta$ be $\phi$-stable. Applying Proposition \ref{prop:open-quotient} to $P_K \subset G$ and then to $P_I \cap L_K \subset L_K$ gives isomorphisms
    \begin{align*}
    X/U_I^F & \xrightarrow{\sim} \left( S_{P_K} / (U_I \cap L_K)^F \right) \times \mathbb{G}_m^{J(P_K)}\\
    & \xrightarrow{\sim} S_{P_I} \times \mathbb{G}_m^{J(P_I) \setminus J(P_K)} \times \mathbb{G}_m^{J(P_K)}.
    \end{align*}
    After regrouping the $\mathbb G_m$-factors by their indices, the composite isomorphism is $\theta_I$.
\end{lemma}

\begin{proof}
The Levi decomposition gives
\begin{equation*}
    U_I = U_K \rtimes (U_I\cap L_K), \qquad U_I^F/U_K^F \simeq (U_I\cap L_K)^F.
\end{equation*}
Under $\theta_K$, the latter group acts naturally on $S_{P_K}$ and acts trivially on $\mathbb{G}_m^{J(P_K)}$. Taking the quotient by $U_K^F$, then by $U_I^F/U_K^F$ successively gives the first isomorphism in Lemma \ref{lem:nested-quotients}, induced by $\theta_K$.

We now apply Proposition~\ref{prop:open-quotient} inside $L_K$ to the parabolic $P_I\cap L_K$, whose standard Levi subgroup is $L_I$. A simple reflection $\tau_K\tau_G(s_j)$ of $c'_K$ is kept in this construction precisely when
\begin{equation*}
    \tau_K \bigl( \tau_K\tau_G(s_j) \bigr) = \tau_G(s_j) \in I.
\end{equation*}
The resulting Coxeter element in $L_I$ is therefore
\begin{equation*}
    \prod_{\substack{j \text{ such that}\\ \tau_G(s_j) \in I}} \tau_I\tau_K \bigl(\tau_K\tau_G(s_j)\bigr),
\end{equation*}
which is nothing but $c_I'$, the product being taken in increasing order of $j$. Thus the proposition gives an isomorphism
\begin{equation*}
    S_{P_K}/(U_I\cap L_K)^F \xrightarrow{\sim} S_{P_I} \times \mathbb G_m^{J(P_I) \setminus J(P_K)}.
\end{equation*}
For $u\in Z_c$, we have
\begin{equation*}
    \pi_K(u)^{-1}F(\pi_K(u)) = \pi_K \left( u^{-1}F(u) \right) = \prod_{\substack{j \text{ such that} \\ \tau_G(s_j) \in K}} x_{\tau_G(\alpha_j)}(a_j(u)).
\end{equation*}
In the construction of Proposition \ref{prop:open-quotient} inside $L_K$, the functions to the $\mathbb G_m$-factors correspond to the root factors killed by $(\pi_I)_{|V_K}$. In the product above, these are precisely the factors with $\tau_G(s_j) \in K \setminus I$, that is $j \in J(P_I) \setminus J(P_K)$. 

On the other hand, we have $(\pi_I)_{|V_K} \circ \pi_K = \pi_I$, from which it follows that 
\begin{equation*}
    \rho_{P_I} = \rho_{P_I \cap L_K} \circ \rho_{P_K},
\end{equation*}
where $\rho_{P_I \cap L_K}: S_{P_K} \to S_{P_I}$ is constructed as above within $L_K$. Thus, up to reordering the $\mathbb G_m$-factors, the composite isomorphism is $\theta_I$.
\end{proof}

Let $P$ be a rational parabolic of type $I$, not necessarily standard, and choose $\gamma \in G^F$ such that $P = \gamma P_I\gamma^{-1}$. Conjugation by $\gamma$ identifies $U_I$ with $U_P$ and induces an isomorphism compatible with Frobenius
\begin{equation*}
    L_I \xrightarrow{\sim} L_P, \qquad \ell \mapsto \gamma \ell \gamma^{-1}U_P.
\end{equation*}
Since $\gamma \in G^F$, conjugation by $\gamma^{-1}$ preserves $X$. Under the identification above, the morphism $B' \mapsto \rho_{P_I}(\gamma^{-1}B'\gamma)$ becomes
\begin{equation*}
    \rho_P: X \to \mathcal B(L_P), \qquad B' \mapsto (B' \cap P)U_P/U_P.
\end{equation*}
We denote its image by $S_P$. It corresponds to $S_{P_I} = X_{L_I}(c'_I)$ under the identification $L_I \simeq L_P$, and is therefore a Coxeter variety for $(L_P, F_P)$. The formula above shows that $\rho_P$ and $S_P$ are independent of the choice of $\gamma$. Moreover, $\rho_P$ is $P^F$-equivariant, with the action on $S_P$ factoring through $P^F/U_P^F \simeq L_P^{F_P}$. By Proposition \ref{prop:open-quotient}, the morphism
\begin{equation*}
    \begin{split}
        X & \to S_P \times \mathbb G_m^{J(P_I)},\\
        B' & \mapsto \left( \rho_P(B'), (a_j(\gamma^{-1}B'\gamma))_{j \in J(P_I)} \right)
    \end{split}
\end{equation*}
induces an isomorphism
\begin{equation*}
    \theta_{P}: X/U_P^F \xrightarrow{\sim} S_P \times \mathbb G_m^{J(P_I)} \simeq S_P \times \mathbb G_m^{d(P)}.
\end{equation*}
While $\rho_{P}$ is independent of $\gamma$, the functions to $\mathbb G_m^{d(P)}$ depend on it by construction. If $P = P_I$, we take $\gamma = 1$ so that $\theta_{P_I} = \theta_I$. For embedded rational parabolics $P \subseteq Q$, we have $U_Q \subseteq U_P$. The group $U_P^F/U_Q^F$ acts on $S_Q$ through its inclusion in $Q^F/U_Q^F \simeq L_Q^{F_Q}$. Choose a rational Borel contained in $P$ and an element $\gamma \in G^F$ carrying $B$ to this Borel. Under conjugation by $\gamma^{-1}$, both $P$ and $Q$ become standard, so that applying Lemma \ref{lem:nested-quotients} and transporting its constructions back gives compatible successive isomorphisms 
\begin{align*}
    X/U_P^F & \xrightarrow{\sim} \left( S_Q / (U_P^F/U_Q^F) \right) \times \mathbb G_m^{d(Q)}\\
    & \xrightarrow{\sim} S_P \times \mathbb G_m^{d(P)-d(Q)} \times \mathbb G_m^{d(Q)},
\end{align*}
whose composition is $\theta_P$ up to reordering of the $\mathbb G_m$-factors.

\subsection{The main theorem}

Our main theorem extends the quotient isomorphisms of Section \ref{subsec:lusztig} from $X$ to the open stars $C_P^*$.

\begin{theorem}\label{thm:star-quotient}
    Let $P$ be a rational parabolic subgroup of $G$. After possibly inverting some of the $\mathbb G_m$-coordinates, the isomorphism $\theta_P: X/U_P^F \simeq S_P \times \mathbb{G}_m^{d(P)}$ extends to an isomorphism
    \begin{equation*}
        C_P^* / U_P^F \xrightarrow{\sim} S_P \times \mathbb{A}^{d(P)}.
    \end{equation*}
    Under this isomorphism, the reduced images of the irreducible components of $C_P^* \setminus X$ are the products of $S_P$ with the coordinate hyperplanes of $\mathbb{A}^{d(P)}$, and the reduced image of $C_P$ is $S_P \times \{0\}$. The induced morphism $C_P^* \to S_P$ extends $\rho_P$ and is $P^F$-equivariant, where $P^F$ acts on $S_P$ through $P^F / U_P^F \simeq L_P^{F_P}$. Its restriction $C_P \to S_P$ is a finite universal homeomorphism.
\end{theorem}

Let us explain the statement a little bit more precisely. Denote by $q_P: C_P^* \to S_P \times \mathbb A^{d(P)}$ the composite of the quotient morphism with the isomorphism of Theorem \ref{thm:star-quotient}. For $P = P_I$, index the coordinates $t_i$ of $\mathbb A^{d(P)}$ by $J(P_I)$. Their pullbacks satisfy
\begin{equation*}
    \forall i \in J(P_I), \qquad t_i \circ (q_P)_{|X} \in \{a_i, a_i^{-1}\}.
\end{equation*}
For each $i \in J(P_I)$, let $M_i$ be the standard rational parabolic whose type is obtained from $\Delta$ by removing the $\phi$-orbit of $\tau_G(s_i)$. These are the maximal proper rational parabolics containing $P_I$. The theorem states that
\begin{equation*}
    \left( q_{P_I}(\overline{C_{M_i}} \cap C_{P_I}^*) \right)_{\mathrm{red}} = S_{P_I} \times \{t_i = 0\} \subseteq S_{P_I} \times \mathbb A^{J(P_I)}.
\end{equation*}
For an arbitrary rational parabolic $P$, this description is transported by the same conjugation used to construct $\theta_P$.

For $P = G$, the theorem gives the identity on $X$. For a rational Borel $B'$, $S_{B'}$ is a point so that $C_{B'}^*/U_{B'}^F \simeq \mathbb A^r$. 

The proof of Theorem \ref{thm:star-quotient} occupies the next Section. In it, we will use the following point counting argument. By \cite{lusztigCoxeterOrbitsEigenspaces1976a} Proposition 1.10, any two Coxeter varieties for the same pair $(G, F)$ have the same number of points over every sufficiently divisible finite extension of $\mathbb F_q$. Since $C_P$ and $S_P$ are Coxeter varieties for $(L_P, F_P)$, there is an integer $\delta \geq 1$ such that
\begin{equation}\label{eq:levi-counts}
    \forall a \geq 1, \qquad |C_P(\mathbb F_{q^{\delta a}})| = |S_P(\mathbb F_{q^{\delta a}})|.
\end{equation}

\section{Unipotent quotients of open stars}\label{sec:geometry}

\subsection{Equivariant rational maps between Coxeter varieties}

We first prove a dominance property for equivariant rational maps between Coxeter varieties.

\begin{lemma}\label{lem:no-boundary-center}
    There is no nonempty irreducible $G^F$-stable closed subset of $\overline X$ contained in $\overline X \setminus X$.
\end{lemma}

\begin{proof}
    Such a subset would be contained in one boundary divisor $\overline{C_M}$ by irreducibility, and hence in every $G^F$-translate of that divisor. There are at least two such translates, and distinct ones have the same type and are disjoint, hence a contradiction.
\end{proof}

When considering different Coxeter elements for the same pair $(G, F)$, we will use explicit notations such as $\overline X_G(c)$ and $C_P(c)$ instead of $\overline X$ and $C_P$.

\begin{proposition}\label{prop:dominance}
    Let $c$ and $c'$ be Coxeter elements for $(G, F)$. Every $G^F$-equivariant rational map
    \begin{equation*}
        X_G(c) \dashrightarrow X_G(c')
    \end{equation*}
    is dominant.
\end{proposition}

\begin{proof}
We argue by induction on $r = |\Delta/\phi|$, simultaneously for all connected reductive groups endowed with an $\mathbb F_q$-structure. For $r = 0$, both varieties are points so that there is nothing to do.

Suppose that $r > 0$, and let $f$ be such a rational map. Fix a maximal proper rational parabolic $M$, and let $\eta_M$ be the generic point of $\overline{C_M(c)}$. Since $\overline X_G(c)$ is smooth, its local ring at $\eta_M$ is a discrete valuation ring. Properness of $\overline X_G(c')$ therefore extends $f$ over this local ring via the valuative criterion, and hence over an open neighbourhood of $\eta_M$. Intersecting its finitely many $M^F$-translates gives an $M^F$-stable neighbourhood $V$ of $\eta_M$. Denote the extended morphism on $V$ by $\widetilde f$. For every $h \in M^F$, the two morphisms $\widetilde f \circ h$ and $h \circ \widetilde f$ from $V$ to $\overline{X}_G(c')$ agree on the open dense subset $V \cap \Omega$, where $\Omega$ is an $M^F$-stable domain of definition for $f$. Since $V$ is integral and $\overline X_G(c')$ is separated, they agree on all of $V$. Thus $\widetilde f$ is $M^F$-equivariant. Restriction to $\overline{C_M(c)}$ then gives an $M^F$-equivariant rational map
\begin{equation*}
    f_M: \overline{C_M(c)} \dashrightarrow\overline X_G(c').
\end{equation*}
Since $U_M^F$ acts trivially on the source, its image lies in 
\begin{equation*}
    \bigl(\overline X_G(c')^{U_M^F}\bigr)_{\mathrm{red}} = \overline{C_M(c')}
\end{equation*}
by \eqref{eq:fixed-locus}. The closure of this image is irreducible and $L_M^{F_M}$-stable. Applying Lemma \ref{lem:no-boundary-center} to $(L_M, F_M)$ shows that it must meet $C_M(c')$. Thus $f_M$ induces an $L_M^{F_M}$-equivariant rational map
\begin{equation*}
    C_M(c) \dashrightarrow C_M(c').
\end{equation*}
Both are Coxeter varieties for $(L_M, F_M)$ whose relative semisimple rank is $r-1$. This map is therefore dominant by induction. 

Let $W$ be the closure of the image of $f$ in $\overline X_G(c')$. The extension $\widetilde f$ near $\eta_M$ also factors through $W$, since $W$ is closed and contains the image of $\Omega$ by $f$. Consequently, $\overline{C_M(c')} \subseteq W$ for every maximal proper rational parabolic $M$. If $f$ were not dominant, then $\dim W \leq r-1$. Since $W$ is irreducible, it would equal each of these divisors, forcing them all to coincide. This contradicts the existence of distinct boundary divisors of the same type.
\end{proof}

\subsection{The case of a maximal rational parabolic}\label{subsec:maximal}

Fix a maximal proper rational parabolic $M$, so that $C_M^* = X \sqcup C_M$. Since $\overline X$ is projective, the finite quotient $\overline X/U_M^F$ exists. The quotient $C_M^*/U_M^F$ is the corresponding open subvariety, and is normal. Let $D$ be the reduced image of $C_M$ in $C_M^*/U_M^F$. The morphism $C_M \to D$ is a finite universal homeomorphism. Indeed, its geometric fibres are $U_M^F$-orbits, and $U_M^F$ fixes $C_M$ pointwise.

Let $t: X \to \mathbb G_m$ be the function for which the composition
\begin{equation*}
    X \to X/U_M^F \xrightarrow{\theta_M} S_M \times \mathbb G_m
\end{equation*}
is $(\rho_M, t)$.

\begin{lemma}\label{lem:levi-extension}
    The morphism $\rho_M: X \to S_M$ extends uniquely to an $M^F$-equivariant morphism $C_M^* \to S_M$, where $M^F$ acts on the target through $M^F/U_M^F \simeq L_M^{F_M}$. Its restriction $C_M \to S_M$ is $L_M^{F_M}$-equivariant and dominant.
\end{lemma}

\begin{proof}
Let $\eta_M$ be the generic point of $C_M$. The local ring $\mathcal O_{C_M^*, \eta_M}$ is a discrete valuation ring. As in the proof of Proposition \ref{prop:dominance}, properness of $\overline{S_M}$ extends $\rho_M$ to a $M^F$-equivariant morphism from a $M^F$-stable open neighbourhood of $\eta_M$ to $\overline{S_M}$.

The closure of the image of $\eta_M$ is irreducible and stable under $L_M^{F_M}$. By Lemma \ref{lem:no-boundary-center} applied to $(L_M, F_M)$, this closure is not contained in $\overline{S_M} \setminus S_M$. Its generic point, which is the image of $\eta_M$, therefore belongs to $S_M$. Consequently, for every regular function $h \in \overline{\mathbb F_q}[S_M]$, the pullback $\rho_M^*h \in \overline{\mathbb F_q}(X) = \overline{\mathbb F_q}(C_M^*)$ is regular on $X$ and at $\eta_M$. Since $C_M^* \setminus X = C_M$, $\eta_M$ is the only point of codimension one of $C_M^*$ outside $X$. Therefore normality of $C_M^*$ implies that $\rho_M^*h$ is regular on all of $C_M^*$. Since $S_M$ is affine, the resulting homomorphism $\overline{\mathbb F_q}[S_M] \to \Gamma(C_M^*, \mathcal O_{C_M^*})$ defines the required extension.

Uniqueness and equivariance of this extension follow by density and separatedness of $S_M$. The action of $M^F$ on $C_M$ factors through $L_M^{F_M}$. Applying Proposition \ref{prop:dominance} to the Coxeter varieties $C_M$ and $S_M$ for $(L_M, F_M)$ proves dominance of the restriction to $C_M$.
\end{proof}

We recall the following corollary of van der Waerden's purity theorem, see \cite{grothendieckElementsGeometrieAlgebriquec} Corollaire 21.12.13.

\begin{lemma}\label{lem:birational-open-immersion}
    Let $f: Y \to Z$ be a birational morphism locally of finite type between integral locally Noetherian schemes, with $Z$ locally factorial. If $f$ is quasi-finite at every point of codimension one of $Y$, then $f$ is a local isomorphism. If $f$ is also separated, then it is an open immersion.
\end{lemma}

\begin{proposition}\label{prop:maximal-product}
    After possibly inverting the $\mathbb G_m$-coordinate, the isomorphism $\theta_M$ extends to an isomorphism of pairs
    \begin{equation*}
        \left( C_M^*/U_M^F, D \right) \xrightarrow{\sim} \left( S_M \times \mathbb A^1, S_M \times \{0\} \right).
    \end{equation*}
\end{proposition}

\begin{proof}
Write $\mathrm{ord}_{C_M}$ for the valuation of $\overline{\mathbb F_q}(X) = \overline{\mathbb F_q}(C_M^*)$ defined by the divisor $C_M$. We first show that $\mathrm{ord}_{C_M}(t) \not = 0$. Otherwise, both $t$ and $t^{-1}$ would be regular on $X$ and at the generic point of $C_M$. Normality would extend them to $C_M^*$, where they remain inverse of each other. Together with Lemma \ref{lem:levi-extension}, this would give a $U_M^F$-invariant morphism descending to
\begin{equation*}
    C_M^*/U_M^F \to S_M \times \mathbb G_m
\end{equation*}
upon taking the quotient. Its restriction to $X/U_M^F$ is $\theta_M$, whose inverse is therefore a section. A section of a separated morphism is a closed immersion. Thus $X/U_M^F$ would be both closed and dense in $C_M^*/U_M^F$, contradicting the fact that $D \not = \emptyset$.

We may therefore replace $t$ by $t^{-1}$ if necessary, and assume that $\mathrm{ord}_{C_M}(t) > 0$. Normality extends $t$ to $C_M^*$, and this extension vanishes on $C_M$. Together with the extension of $\rho_M$, this gives a $U_M^F$-invariant morphism descending to
\begin{equation*}
    f: C_M^*/U_M^F \to S_M \times \mathbb A^1
\end{equation*}
upon taking the quotient. It is an isomorphism over $S_M \times \mathbb G_m$ and it maps $D$ dominantly to $S_M \times \{0\}$ by Lemma \ref{lem:levi-extension}. Since $\dim D = \dim S_M$, this restriction to $D$ is generically finite. Hence $f$ is quasi-finite at the generic point of $D$. It is also quasi-finite at every other point of codimension one, which all lie in $X/U_M^F$. The morphism $f$ is birational and separated, and its target is smooth hence locally factorial. Therefore, Lemma \ref{lem:birational-open-immersion} shows that $f$ is an open immersion.

Since $f$ is an open immersion and $f^{-1}(S_M \times \mathbb G_m) = X/U_M^F$, its restriction identifies $D$ with an open subvariety of $S_M \times \{0\} \simeq S_M$. Write this subvariety as $S_0 \subseteq S_M$. We therefore obtain an isomorphism $D \xrightarrow{\sim} S_0 \times \{0\} \simeq S_0$. Composing with the finite universal homeomorphism $C_M \to D$ gives a finite universal homeomorphism $C_M \to S_0$.
 
Choose $\delta \geq 1$ sufficiently divisible so that the diagram $C_M \to S_0 \hookrightarrow S_M$ is defined over $\mathbb F_{q^{\delta}}$, and such that the point counting \eqref{eq:levi-counts} holds. A finite universal homeomorphism induces a bijection on points over every finite field, so
\begin{equation*}
    \forall a \geq 1, \qquad |S_0(\mathbb F_{q^{\delta a}})| = |C_M(\mathbb F_{q^{\delta a}})| = |S_M(\mathbb F_{q^{\delta a}})|.
\end{equation*}
The complement $S_M \setminus S_0$ has no point over any finite extension of $\mathbb F_{q^\delta}$, and is therefore empty. Thus $f$ is an isomorphism, identifying $D$ with $S_M \times \{0\}$.
\end{proof}

\subsection{The case of an arbitrary rational parabolic}

\begin{proof}[Proof of Theorem \ref{thm:star-quotient}]
The case $P = G$ is immediate. For proper $P$, use the conjugation defining $\theta_P$ to reduce to the case where $P = P_I$ is standard. We use the notations developed in Section \ref{subsec:lusztig}. For each $i \in J(P_I)$, let $M_i$ be the standard rational parabolic whose type is obtained from $\Delta$ by removing the $\phi$-orbit of $\tau_G(s_i)$. These are the maximal proper rational parabolics containing $P_I$. Proposition \ref{prop:maximal-product} gives, for every $i$, a choice $b_i \in \{a_i, a_i^{-1}\}$ and an isomorphism
\begin{equation*}
    C_{M_i}^*/U_{M_i}^F \xrightarrow{\sim} S_{M_i} \times \mathbb A^1
\end{equation*}
whose restriction to $X/U_{M_i}^F$ is induced by $(\rho_{M_i}, b_i): X \to S_{M_i} \times \mathbb G_m$.

We will next take the quotient by $U_I^F/U_{M_i}^F$. Since $U_I^F \subseteq U^F$, its action fixes the function $b_i$ on the dense open subset $X/U_{M_i}^F$, and hence fixes the affine coordinate on the extension. The first projection is also equivariant by Lemma \ref{lem:levi-extension}. Thus $U_I^F/U_{M_i}^F$ acts naturally on $S_{M_i}$ and trivially on $\mathbb A^1$. Applying Lemma \ref{lem:nested-quotients}, we obtain
\begin{equation}\label{eq:star-codimension-one}
    \begin{split}
        C_{M_i}^*/U_I^F & \xrightarrow{\sim} \left( S_{M_i}/(U_I^F/U_{M_i}^F) \right) \times \mathbb A^1\\
        & \xrightarrow{\sim} S_{P_I} \times \mathbb G_m^{J(P_I)\setminus\{i\}} \times \mathbb A^1.
    \end{split}
\end{equation}
In the second isomorphism, we use the $\mathbb G_m$-coordinates $b_j \in \{a_j, a_j^{-1}\}$ chosen above, for $j \in J(P_I) \setminus \{i\}$. The compatibility in Lemma \ref{lem:nested-quotients} shows that the restriction to $X/U_I^F$ is induced by
\begin{equation*}
    \left( \rho_{P_I}, (b_j)_{j \in J(P_I) \setminus \{i\}}, b_i \right): X \to S_{P_I} \times \mathbb G_m^{J(P_I) \setminus \{i\}} \times \mathbb G_m.
\end{equation*}
For each $i \in J(P_I)$, the open subset $C_{M_i}^*/U_I^F$ contains $X/U_I^F$ and the generic point of the reduced image of $\overline{C_{M_i}} \cap C_{P_I}^*$. Their union therefore contains every point of codimension one of $C_{P_I}^*/U_I^F$. After the natural inclusions
\begin{equation}\label{eq:open-immersion-f}
    S_{P_I} \times \mathbb G_m^{J(P_I) \setminus \{i\}} \times \mathbb A^1 \hookrightarrow S_{P_I} \times \mathbb A^{J(P_I)},
\end{equation}
the morphisms in \eqref{eq:star-codimension-one} agree on $X/U_I^F$. Consequently, the pullback of every regular function on $S_{P_I} \times \mathbb A^{J(P_I)}$ is regular at every point of codimension one of the quotient. Normality of $C_{P_I}^*/U_I^F$ and affineness of the target space give a morphism
\begin{equation*}
    f: C_{P_I}^*/U_I^F \to S_{P_I} \times \mathbb A^{J(P_I)}
\end{equation*}
extending the isomorphism on $X/U_I^F$ induced by $(\rho_{P_I}, (b_j)_{j\in J(P_I)})$.

On each $C_{M_i}^*/U_I^F$, the morphism $f$ is the isomorphism \eqref{eq:star-codimension-one} followed by the open immersion \eqref{eq:open-immersion-f}. Hence $f$ is quasi-finite at every point of codimension one. It is birational and separated, with smooth target, so Lemma \ref{lem:birational-open-immersion} shows that $f$ is an open immersion.

We prove surjectivity again by counting points. For each rational parabolic $Q$ containing $P_I$, let $D_Q$ be the reduced image of $C_Q$ in $C_{P_I}^*/U_I^F$. The stratum $C_Q$ is $U_I^F$-stable and closed in the stable open subset $C_Q^* \subseteq C_{P_I}^*$. Thus $D_Q$ is locally closed, and $C_Q$ is its reduced inverse image. In particular, $C_Q \to D_Q$ is finite. The induced morphism $C_Q/U_I^F \to D_Q$ is also finite. Indeed, over an affine open $\mathrm{Spec}(A) \subseteq D_Q$ with inverse image $\mathrm{Spec}(B) \subseteq C_Q$, the inclusions $A \subseteq B^{U_I^F} \subseteq B$ show that $B^{U_I^F}$ is a finite $A$-module.

 The geometric fibers of the quotient morphism $C_Q \to C_Q/U_I^F$ and of the morphism $C_Q \to D_Q$ are the $U_I^F$-orbits in both cases. Therefore the induced finite morphism $C_Q/U_I^F \to D_Q$ is surjective and universally injective, hence a universal homeomorphism.

Since $U_Q^F$ acts trivially on $C_Q$, the action of $U_I^F$ on this stratum factors through $U_I^F/U_Q^F \simeq (U_I/U_Q)^{F_Q}$. The group $U_I/U_Q$ is the unipotent radical of the rational parabolic $P_I/U_Q \subseteq L_Q$, whose Levi quotient identifies canonically with $L_{P_I}$ as follows
\begin{equation*}
    (P_I/U_Q) / (U_I/U_Q) \xrightarrow{\sim} P_I/U_I = L_{P_I}.
\end{equation*}
Applying Proposition \ref{prop:open-quotient} to the Coxeter variety $C_Q$ for $(L_Q, F_Q)$ and the parabolic $P_I/U_Q$ therefore gives an isomorphism
\begin{equation*}
    C_Q/U_I^F \xrightarrow{\sim} S_Q^{L_{P_I}} \times \mathbb G_m^{r(Q)-r(P_I)},
\end{equation*}
where $S_Q^{L_{P_I}}$ denotes the Coxeter variety for $(L_{P_I}, F_{P_I})$ obtained by this construction. Its defining Coxeter element need not coincide with that of $S_{P_I}$. 

Choose $\delta \geq 1$ sufficiently divisible such that $f$ and all the preceding strata and morphisms are defined over $\mathbb F_{q^\delta}$, and such that the point counting \eqref{eq:levi-counts} holds for all the Coxeter varieties for $(L_{P_I}, F_{P_I})$ occurring here. Since finite universal homeomorphisms induce bijections on points over every finite field, we obtain
\begin{equation*}
    \forall a \geq 1, \quad |D_Q(\mathbb F_{q^{\delta a}})| = |S_Q^{L_{P_I}}(\mathbb F_{q^{\delta a}})| (q^{\delta a}-1)^{r(Q)-r(P_I)} = |S_{P_I}(\mathbb F_{q^{\delta a}})| (q^{\delta a}-1)^{r(Q)-r(P_I)}.
\end{equation*}
The types of the rational parabolics containing $P_I$ are obtained by adjoining some of the $d(P_I)$ remaining $\phi$-orbits to $I$. Consequently,
\begin{equation*}
    \left| \{Q \supseteq P_I \mid r(Q) - r(P_I) = j\} \right| = \binom{d(P_I)}{j}.
\end{equation*}
The locally closed subsets $D_Q$ partition the quotient, so
\begin{equation*}
    \begin{split}
        \left| (C_{P_I}^*/U_I^F)(\mathbb F_{q^{\delta a}}) \right| & = \sum_{Q \supseteq P_I} |D_Q(\mathbb F_{q^{\delta a}})| \\
        & = |S_{P_I}(\mathbb F_{q^{\delta a}})| \sum_{j = 0}^{d(P_I)} \binom{d(P_I)}{j}(q^{\delta a}-1)^j\\
        & = |S_{P_I}(\mathbb F_{q^{\delta a}})| q^{\delta a d(P_I)}.
    \end{split}
\end{equation*}
This is the number of $\mathbb F_{q^{\delta a}}$-points of $S_{P_I} \times \mathbb A^{d(P_I)}$. The complement of the open image of $f$ therefore has no point over any finite extension of $\mathbb F_{q^\delta}$, so it is empty. Thus $f$ is an isomorphism.

Let $t_i$ be the affine coordinate indexed by $i \in J(P_I)$. Equation \eqref{eq:star-codimension-one} gives
\begin{equation*}
    f(D_{M_i}) = S_{P_I} \times \mathbb G_m^{J(P_I) \setminus \{i\}} \times \{0\}.
\end{equation*}
Taking closures identifies the reduced image of $\overline{C_{M_i}} \cap C_{P_I}^*$ with $S_{P_I} \times \{t_i = 0\}$. The divisors $\overline{C_{M_i}} \cap C_{P_I}^*$ are $U_{P_I}^F$-stable, and their intersection is $C_{P_I}$. Since the fibres of the quotient morphisms are orbits, $D_{P_I}$ is the reduced intersection of their images. It follows that
\begin{equation*}
    f(D_{P_I}) = S_{P_I} \times \{0\}.
\end{equation*}
As $U_I^F$ acts trivially on $C_{P_I}$, the finite universal homeomorphism $C_{P_I}/U_I^F \to D_{P_I}$ gives the required finite universal homeomorphism $C_{P_I} \to S_{P_I}$.

Finally, composing $f$ with the quotient morphism gives $q_{P_I}: C_{P_I}^* \to S_{P_I} \times \mathbb A^{d(P_I)}$. Its projection to $S_{P_I}$ extends $\rho_{P_I}$ by construction. This projection is $P_I^F$-equivariant, with the action on $S_{P_I}$ factoring through $L_{P_I}^{F_{P_I}}$, because this holds on the dense open subset $X$ and $S_{P_I}$ is separated.
\end{proof}

\begin{corollary}\label{cor:star-affine-flat}
    Every open star $C_P^*$ is affine. The morphism
    \begin{equation*}
        q_P: C_P^* \to S_P \times \mathbb A^{d(P)}
    \end{equation*}
    is finite flat of degree $|U_P^F|$.
\end{corollary}

\begin{proof}
The morphism $q_P$ is finite and its target is affine, so $C_P^*$ is affine. Both varieties are smooth of the same dimension, and $q_P$ is surjective. By ``miracle flatness'', see \cite{stacks-project} Tag 00R4 for instance, $q_P$ is flat. Over $S_P \times \mathbb G_m^{d(P)}$, the morphism is the $U_P^F$-torsor $X \to X/U_P^F$, which gives the required degree.
\end{proof}

\section{Wang's conjecture}\label{sec:restriction}

Fix a prime $\ell \not = p$ and an integer $a \geq 1$. We write $\Lambda \coloneq \mathbb Z/\ell^a\mathbb Z$, and also $R\Gamma$ and $R\Gamma_c$ respectively for ordinary étale cohomology and for étale cohomology with compact support. Since $p$ is invertible in $\Lambda$, taking invariants of $\Lambda$-modules under a finite $p$-group is exact, and we shall use this fact repeatedly in this section. 

\subsection{Restriction morphisms and reduction to maximal parabolics}\label{subsec:coxeter-cover}

Choose a representative $\dot c \in N_G(T)$ of $c$. The Deligne-Lusztig covering
\begin{equation*}
    Y = Y(\dot c) \coloneq \{gU \in G/U \mid g^{-1}F(g) \in U \dot c U\} \xrightarrow{\pi} X
\end{equation*}
is a torsor under the finite abelian group
\begin{equation*}
    T^{cF} = \{t \in T \mid \dot cF(t)\dot c^{-1} = t\},
\end{equation*}
whose order is prime to $p$. Its right $T^{cF}$-action commutes with the left $G^F$-action. Let $\overline Y$ be the normalization of $\overline X$ in $Y$, and denote the resulting finite morphism by $\overline{\pi}: \overline Y \to \overline X$. Both actions of $T^{cF}$ and of $G^F$ extend to $\overline Y$, see \cite{bonnafeCompactificationVarietesDeligneLusztig2009} Theorem 1.2. We write
\begin{equation*}
    j: X \hookrightarrow \overline X, \qquad j': Y\hookrightarrow \overline Y.
\end{equation*}
For every proper rational parabolic $P$, we write
\begin{equation*}
    E_P \coloneq \left( \overline{\pi}^{-1}(C_P) \right)_{\mathrm{red}},
\end{equation*}
and denote the natural immersions by
\begin{equation*}
    i_P: \overline{C_P} \hookrightarrow \overline X, \qquad j'_P: E_P \hookrightarrow \overline Y.
\end{equation*}
Restriction along $i_P$ gives the canonical morphism
\begin{equation*}
    R\Gamma(X, \pi_*\Lambda) \simeq R\Gamma(\overline X, Rj_*\pi_*\Lambda) \xrightarrow{\mathrm{res}} R\Gamma(\overline{C_P}, i_P^*Rj_*\pi_*\Lambda).
\end{equation*}
Restricting it to the $U_P^F$-invariants in its source gives
\begin{equation}\label{eq:wang-closed}
    R\Gamma(X, \pi_*\Lambda)^{U_P^F} \to R\Gamma(\overline{C_P}, i_P^*Rj_*\pi_*\Lambda).
\end{equation}

\begin{theorem}\label{thm:wang}
    For every proper rational parabolic $P$, the morphism \eqref{eq:wang-closed} is an isomorphism. It is compatible with the natural actions of $P^F$ and $T^{cF}$.
\end{theorem}

The action of $P^F$ on both sides of \eqref{eq:wang-closed} is the obvious one, whereas the $T^{cF}$-action is induced by the action of $T^{cF}$ on $\pi_*\Lambda$. There is a second formulation of Theorem \ref{thm:wang} on $Y$ as follows. Restriction along $j'_P$ gives
\begin{equation*}
    R\Gamma(Y, \Lambda) \simeq R\Gamma(\overline Y, Rj'_*\Lambda) \xrightarrow{\mathrm{res}} R\Gamma(E_P, j_P'^*Rj'_*\Lambda),
\end{equation*}
and hence a morphism
\begin{equation}\label{eq:wang-open}
    R\Gamma(Y, \Lambda)^{U_P^F} \to R\Gamma(E_P, j_P'^*Rj'_*\Lambda).
\end{equation}

\begin{theorem}\label{thm:wang-open}
    For every proper rational parabolic $P$, the morphism \eqref{eq:wang-open} is an isomorphism. It is compatible with the actions of $P^F$ and $T^{cF}$.
\end{theorem}

Theorems \ref{thm:wang} and \ref{thm:wang-open} are respectively the Conjectures A and A' of \cite{wangCohomologieCompactificationVarietes2014}. The equivalence between both statements is proved in \cite{wangCohomologieCompactificationVarietes2014} Lemmas 3.2.1 and 3.2.2. It is therefore enough to prove Theorem \ref{thm:wang-open}. 

\begin{remark}
    Wang proved Theorems \ref{thm:wang} and \ref{thm:wang-open} for $G = \mathrm{GL}_d$ with the standard Frobenius $F((a_{ij})) = (a_{ij}^q)$ and the Coxeter element $c = (1 \ \ldots \ d) \in \mathfrak S_d$, see \cite{wangCohomologieCompactificationVarietes2014} Theorem 5.1.1. A crucial step in his proof is the quotient isomorphism of Theorem \ref{thm:star-quotient} for \textit{maximal} proper rational parabolics. He establishes this isomorphism in \cite{wangCohomologieCompactificationVarietes2014} Section 4.2 using the description of $\overline X$ as a succession of blow-ups of a projective space. Our quotient theorem provides the corresponding geometric input for arbitrary pairs $(G, F)$.
\end{remark}

We also use Wang's reduction to maximal proper rational parabolics. 

\begin{lemma}\label{lem:wang-reduction}
    Suppose that, for every rational parabolic $P$ of $G$, Theorem \ref{thm:wang-open} holds for any Deligne-Lusztig covering of any Coxeter variety for $(L_P, F_P)$, with respect to every maximal proper rational parabolic of $L_P$. Then it holds for every Deligne-Lusztig covering of every Coxeter variety for $(G, F)$, with respect to every proper rational parabolic of $G$.
\end{lemma}

In the hypothesis of this lemma, the case $P = G$ is included. The reduction is proved in \cite{wangCohomologieCompactificationVarietes2014} Section 3.3, with a minor correction needed in the nonsplit case. Indeed, the proof there requires a $\phi$-stable subset $I \subset \Delta$ with $|I| = |\Delta|-1$, which need not exist. This is remedied by indexing the induction by the relative semisimple rank $|\Delta/\phi|$ instead. For a nonmaximal proper $\phi$-stable subset $J \subset \Delta$, one can choose a maximal proper $\phi$-stable subset $I$ containing $J$, and then 
\begin{equation*}
    |I/\phi| = |\Delta/\phi|-1,
\end{equation*}
since $\Delta \setminus I$ is a single $\phi$-orbit. With this adjustment, the remainder of Wang's argument applies unchanged.

In virtue of Lemma \ref{lem:wang-reduction}, it is enough to prove Theorem \ref{thm:wang-open} when $P = M$ is a maximal proper rational parabolic of $G$, uniformly for all pairs $(G, F)$.

\subsection{An auxiliary covering}\label{subsec:wang-maximal}

We adapt the proof of \cite{wangCohomologieCompactificationVarietes2014} Section 5.2. We first construct an auxiliary covering whose unipotent quotient is a product. On this product, the cohomological comparison is made with constant coefficients and reduces to the Künneth formula. We then descend the resulting isomorphism to $Y$, following \cite{wangCohomologieCompactificationVarietes2014} Section 5.2.3. 

Let $M$ be a maximal proper rational parabolic of $G$. By conjugation under $G^F$, we may assume that $M$ is standard. Write $J(M) = \{i\}$. Theorem \ref{thm:star-quotient} gives
\begin{equation*}
    C_M^*/U_M^F \simeq S_M \times \mathbb A^1, \qquad X/U_M^F \simeq S_M \times \mathbb G_m.
\end{equation*}
On the second factor, we use the coordinate whose pullback to $X$ is the function $b_i \in \{a_i,a_i^{-1}\}$ selected by the theorem. Accordingly, on the isomorphism $X/U^F \simeq \mathbb G_m^r$ afforded by Proposition \ref{prop:open-quotient}, we use $b_i$ as the $i$-th coordinate and $a_j$ for $j \not = i$. Lemma \ref{lem:nested-quotients} gives a commutative diagram
\begin{equation}\label{eq:commutative-diagram}
    \begin{tikzcd}
        X/U_M^F \arrow[r] \arrow[d, "\sim", sloped] & X/U^F \arrow[d, "\sim", sloped]\\
        S_M \times \mathbb G_m \arrow[r] & \mathbb G_m^{\{j \not = i\}} \times \mathbb G_m
    \end{tikzcd}
\end{equation}
The top map is the quotient by $U^F/U_M^F$. The bottom map is the product of
\begin{equation*}
    S_M \to S_M/ (U^F/U_M^F) \simeq \mathbb G_m^{\{j \not = i\}}
\end{equation*}
with the identity on $\mathbb G_m$. 

Choose a connected component $Y^0$ of $Y$, and let $ H\subseteq T^{cF}$ be its stabilizer. Since $X$ is connected, $T^{cF}$ acts transitively on the connected components of $Y$. The restriction $Y^0 \to X$ is an $H$-torsor, and
\begin{equation*}
    Y \simeq Y^0 \times^H T^{cF}.
\end{equation*}
In particular, the connected components of $Y$ are indexed by $T^{cF}/H$.

The action of $U^F$ on $Y$ commutes with $T^{cF}$. Its action on the set of connected components therefore defines a homomorphism
\begin{equation*}
    U^F \to \mathrm{Aut}_{T^{cF}}(T^{cF}/H) \simeq T^{cF}/H.
\end{equation*}
Since $U^F$ is a $p$-group and $T^{cF}/H$ has order prime to $p$, this homomorphism is trivial. Thus $U^F$ preserves every connected component of $Y$, and in particular acts on $Y^0$. Since $U^F$ acts freely on $X$, passing to $U^F$-quotients gives an $H$-torsor
\begin{equation*}
    Y^0/U^F \to X/U^F \simeq \mathbb G_m^r.
\end{equation*}
Its source is connected, being the image of the connected variety $Y^0$ under the quotient morphism.

We recall the covering construction of \cite{bonnafe_coxeter_2006} Section 3.2.2, also used in \cite{wangCohomologieCompactificationVarietes2014} Section 5.2. The maximal prime-to-$p$ quotient of the geometric fundamental group of $\mathbb G_m^r$ satisfies
\begin{equation*}
    \pi_1^{(p')}(\mathbb G_m^r) \simeq \prod_{j=1}^r \pi_1^{(p')}(\mathbb G_m),
\end{equation*}
where each factor is procyclic. The connected $H$-torsor above corresponds to a continuous surjection
\begin{equation*}
    \prod_{j=1}^r \pi_1^{(p')}(\mathbb G_m) \twoheadrightarrow H.
\end{equation*}
Let $m_j$ be the order of the image of the $j$-th factor. Then $m_j$ divides $|H|$ and is prime to $p$. The restriction to this factor induces an embedding $\mu_{m_j} \hookrightarrow H$, where the quotient $\pi_1^{(p')}(\mathbb G_m)\twoheadrightarrow\mu_{m_j}$ corresponds to the Kummer covering $w_j \mapsto w_j^{m_j}$. Since $H$ is abelian, these embeddings give an exact sequence
\begin{equation*}
    \{1\} \to N \to \prod_{j = 1}^r \mu_{m_j} \to H \to \{1\},
\end{equation*}
where $N$ denotes the kernel. Since each restriction $\mu_{m_j} \to H$ is injective,
\begin{equation}\label{eq:auxiliary-kernel}
    \forall 1 \leq j \leq r, \qquad N \cap \mu_{m_j} = \{1\},
\end{equation}
where $\mu_{m_j}$ is viewed as the $j$-th factor of $\prod_{j=1}^r \mu_{m_j}$. The original surjection to $H$ thus factors through the quotient $\pi_1^{(p')}(\mathbb G_m^r) \twoheadrightarrow \prod_{j=1}^r \mu_{m_j}$ associated with the power map
\begin{equation*}
    \mathbb G_m^r \to \mathbb G_m^r, \qquad (w_j)_{1 \leq j \leq r} \mapsto (w_j^{m_j})_{1 \leq j \leq r}.
\end{equation*}
The Galois correspondence therefore gives a factorization of this map
\begin{equation}\label{eq:auxiliary-torus-cover}
    \mathbb G_m^r \to Y^0/U^F \to \mathbb G_m^r,
\end{equation}
whose first morphism is an $N$-torsor. 

Since $U^F$ acts freely on $X$, it also acts freely on $Y^0$. Thus both quotient morphisms $X \to X/U^F$ and $Y^0 \to Y^0/U^F$ are $U^F$-torsors. The natural morphism
\begin{equation*}
    Y^0 \to X \times_{X/U^F} (Y^0/U^F), \qquad y \mapsto (\pi(y), [y]),
\end{equation*}
where $[y]$ denotes the image of $y$ in $Y^0/U^F$, is $U^F$-equivariant over $Y^0/U^F$. Its target is also a $U^F$-torsor over this base by base change from $X \to X/U^F$. It is therefore an isomorphism.

Using \eqref{eq:auxiliary-torus-cover} and the preceding identification, we obtain a morphism
\begin{equation*}
    \pi^1: Y^1 \coloneq X \times_{\mathbb G_m^r} \mathbb G_m^r \to Y^0,
\end{equation*}
where the morphism $\mathbb G_m^r \to \mathbb G_m^r$ in the fibre product is the power map. The factorization \eqref{eq:auxiliary-torus-cover} gives a commutative diagram
\begin{equation*}
    \begin{tikzcd}
        Y^1 \arrow[r, "\pi^1"] \arrow[d] & Y^0 \arrow[r, "\pi"] \arrow[d] & X \arrow[d]\\
        \mathbb G_m^r \arrow[r] & Y^0/U^F \arrow[r] & X/U^F \simeq \mathbb G_m^r.
    \end{tikzcd}
\end{equation*}
The right-hand square is cartesian by the preceding argument. The outer rectangle is cartesian by the definition of $Y^1$, since the composition of the lower horizontal arrows is the power map. Hence the left-hand square is also cartesian. Its lower horizontal morphism is an $N$-torsor, so $\pi^1$ is an $N$-torsor as well. The actions of $U^F$ on $X$ and of $\prod_{j = 1}^r \mu_{m_j}$ on the second factor induce commuting actions on $Y^1$. Define also
\begin{equation*}
    S_M^1 \coloneq S_M \times_{\mathbb G_m^{\{j \not = i\}}} \mathbb G_m^{\{j \not = i\}},
\end{equation*}
using the power maps $w_j \mapsto w_j^{m_j}$ for $j \not = i$. The morphism $S_M^1 \to S_M$ is finite étale, so $S_M^1$ is smooth and affine. By the commutative diagram \eqref{eq:commutative-diagram} and the definitions of $Y^1$ and $S_M^1$, the square
\begin{equation*}
    \begin{tikzcd}
        Y^1 \arrow[r] \arrow[d] & S_M^1 \times \mathbb G_m \arrow[d]\\
        X \arrow[r, "{(\rho_M,b_i)}"] & S_M \times \mathbb G_m
    \end{tikzcd}
\end{equation*}
is cartesian, where the right-hand morphism is the product of $S_M^1 \to S_M$ and $w_i \mapsto w_i^{m_i}$. The lower horizontal morphism is a $U_M^F$-torsor, since it identifies with $X \to X/U_M^F$. Its base change is therefore a $U_M^F$-torsor as well, giving an $N$-equivariant isomorphism
\begin{equation}\label{eq:auxiliary-open-product}
    Y^1/U_M^F \simeq S_M^1 \times \mathbb G_m.
\end{equation}
Let $\overline Y^0$ be the normalization of $C_M^*$ in $Y^0$. It comes together with a finite morphism $\overline Y^0 \to C_M^*$. Define then
\begin{equation*}
    E^0 \coloneq \left(\overline Y^0 \times_{C_M^*} C_M\right)_{\mathrm{red}}.
\end{equation*} 
Thus $\overline Y^0$ is an open and closed subvariety of $\overline{\pi}^{-1}(C_M^*)$. The commuting actions of $U^F$ and $H$ on $Y^0$ extend to $\overline Y^0$ by functoriality of normalization. By \cite{bonnafeCompactificationVarietesDeligneLusztig2009} Theorem 1.2 (a) and (d), the variety $\overline Y^0$ is smooth and $E^0$ is a smooth divisor. The following lemma adapts the constructions of \cite{wangCohomologieCompactificationVarietes2014} Section 5.2. 

\begin{lemma}\label{lem:auxiliary-cover-extension}
    There exist a smooth variety $\overline Y^1$ containing $Y^1$ as a dense open subvariety, and a finite étale $N$-torsor
    \begin{equation*}
        \overline{\pi}^1: \overline Y^1 \to \overline Y^0
    \end{equation*}
    whose restriction over $Y^0$ agrees with $\pi^1: Y^1 \to Y^0$. The action of $U_M^F$ on $Y^1$ extends to $\overline Y^1$ and commutes with the $N$-action. Put $E^1 \coloneq (\overline{\pi}^1)^{-1}(E^0)$. Then $E^1$ is a smooth divisor and $U_M^F$ fixes $E^1$ pointwise. Moreover, \eqref{eq:auxiliary-open-product} extends to an isomorphism
    \begin{equation*}
        \overline Y^1/U_M^F \simeq S_M^1 \times \mathbb A^1,
    \end{equation*}
    under which the reduced image of $E^1$ is $S_M^1 \times \{0\}$.
\end{lemma}

\begin{proof}
Under the isomorphism \eqref{eq:auxiliary-open-product}, the group $\prod_{j \not = i} \mu_{m_j}$ acts on $S_M^1$ through the covering $S_M^1 \to S_M$, while $\mu_{m_i}$ acts on the second factor $\mathbb G_m$ by multiplication. The latter action extends to $\mathbb A^1$, giving an action of $\prod_j \mu_{m_j}$ on $S_M^1 \times \mathbb A^1$. The subgroup $N$ acts freely on this product. Indeed, $\prod_{j\not = i} \mu_{m_j}$ acts freely on $S_M^1$, and by \eqref{eq:auxiliary-kernel} the projection
\begin{equation*}
    N \to \prod_{j \not = i} \mu_{m_j}
\end{equation*}
is injective. Hence $S_M^1 \times \mathbb A^1 \to (S_M^1 \times \mathbb A^1)/N$ is a finite étale $N$-torsor.

Both $(S_M^1 \times \mathbb A^1)/N$ and $\overline Y^0/U_M^F$ are normal and finite over $S_M \times \mathbb A^1$. Over $S_M \times \mathbb G_m$, the first variety restricts to
\begin{equation*}
    (S_M^1 \times \mathbb G_m)/N \simeq (Y^1/U_M^F)/N \simeq (Y^1/N)/U_M^F \simeq Y^0/U_M^F,
\end{equation*}
using the $N$-equivariance of \eqref{eq:auxiliary-open-product}, the commuting actions of $N$ and $U_M^F$, and the identification $Y^1/N \simeq Y^0$. The second variety also restricts to $Y^0/U_M^F$ by the definition of $\overline Y^0$. By uniqueness of the normalization, the two varieties are therefore isomorphic over $S_M \times \mathbb A^1$. Define $\overline Y^1$ by the cartesian square
\begin{equation*}
    \begin{tikzcd}
        \overline Y^1 \arrow[r] \arrow[d,"\overline{\pi}^1"] & S_M^1 \times \mathbb A^1 \arrow[d]\\
        \overline Y^0 \arrow[r] & \overline Y^0/U_M^F.
    \end{tikzcd}
\end{equation*}
The left vertical arrow is a finite étale $N$-torsor whose restriction over $Y^0$ is $\pi^1$. The action of $U_M^F$ through the factor $\overline Y^0$ extends its action on $Y^1$ and commutes with the action of $N$ through the other factor. Since the right vertical arrow is finite étale, quotients by finite groups commute with this base change by \cite{stacks-project} Tag 0BRH. Thus the upper horizontal arrow induces an isomorphism
\begin{equation}\label{eq:auxiliary-closed-product}
    \overline Y^1/U_M^F \xrightarrow{\sim} S_M^1 \times \mathbb A^1,
\end{equation}
extending \eqref{eq:auxiliary-open-product}. The reduced image of $E^1$ is the zero section by construction. Smoothness of $\overline Y^1$ and of its divisor $E^1$ follows by étale pullback from $(\overline Y^0, E^0)$.

Finally, $\overline Y^0/H = C_M^*$ so the geometric fibres of $E^0 \to C_M$ are $H$-orbits. The group $U_M^F$ fixes $C_M$ pointwise and acts on each such fibre by $H$-equivariant automorphisms. Since $H$ is abelian of order prime to $p$, the automorphism group of an $H$-orbit also has order prime to $p$. Thus $U_M^F$ fixes every geometric point of $E^0$, and acts trivially on $E^0$ because this variety is reduced. In the cartesian square above, $U_M^F$ acts through the $\overline Y^0$ factor. It therefore fixes $E^1$ pointwise as well.
\end{proof}

\subsection{Proof of the conjecture}

\begin{proposition}\label{prop:wang-maximal}
    The morphism \eqref{eq:wang-open} is an isomorphism for every maximal proper rational parabolic $M$ of $G$.
\end{proposition}

\begin{proof}
We first consider the standard parabolic $M$ fixed in Section \ref{subsec:wang-maximal}. For $\varepsilon \in \{0, 1\}$, write
\begin{equation*}
    j_{\varepsilon}: Y^{\varepsilon} \hookrightarrow \overline Y^{\varepsilon}, \qquad i_{\varepsilon}: E^{\varepsilon} \hookrightarrow \overline Y^{\varepsilon}.
\end{equation*}
Following \cite{wangCohomologieCompactificationVarietes2014} Section 5.2.2, we first prove that restriction induces an isomorphism
\begin{equation}\label{eq:auxiliary-restriction}
    R\Gamma(Y^1, \Lambda)^{U_M^F} \xrightarrow{\sim} R\Gamma(E^1, i_1^*Rj_{1,*}\Lambda).
\end{equation}
We will then descend this comparison to $Y$.

Using \eqref{eq:auxiliary-closed-product}, denote the quotient morphism and the immersions by
\begin{equation*}
    \begin{split}
        q_1&: \overline Y^1 \to S_M^1 \times \mathbb A^1,\\
        j_{\times}&: S_M^1 \times \mathbb G_m \hookrightarrow S_M^1 \times \mathbb A^1,\\
        i_{\times}&: S_M^1 \hookrightarrow S_M^1 \times \mathbb A^1,
    \end{split}
\end{equation*}
where $i_{\times}$ is the zero section. We regard $(q_1)_{|Y^1}$ and $(q_1)_{|E^1}$ as morphisms to $S_M^1 \times \mathbb G_m$ and to $S_M^1$ respectively.

We first show that taking $U_M^F$-invariants does not change the boundary complex $i_1^*Rj_{1,*}\Lambda$. Since $(\overline Y^1, E^1)$ is a smooth pair, cohomological purity gives
\begin{equation*}
    i_1^*R^m j_{1,*}\Lambda \simeq \begin{cases}
        \Lambda & \text{if } m = 0,\\
        \Lambda(-1) & \text{if } m = 1,\\
        0 & \text{else},
    \end{cases}
\end{equation*}
see \cite{artinTheorieToposCohomologie1973} Exposé XVI Théorème 3.7, and Exposé XIX Paragraphs (3.3.1) and (3.3.2). The action of $U_M^F$ on both nonzero sheaves is trivial. Indeed, in degree zero this follows from its trivial action on $E^1$. In degree one, up to the Tate twist, the identification is given by the Kummer class of a local equation of $E^1$. An element of $U_M^F$ changes this equation by a unit. In a strictly henselian local ring, this unit has an $\ell^a$-th root, so its Kummer class vanishes. Thus the class of the local equation is unchanged. Exactness of $U_M^F$-invariants now shows that the natural morphism
\begin{equation}\label{eq:boundary-invariants}
    (i_1^*Rj_{1,*}\Lambda)^{U_M^F} \xrightarrow{\sim} i_1^*Rj_{1,*}\Lambda,
\end{equation}
is an isomorphism.

The morphism $(q_1)_{|Y^1}$ is a finite étale $U_M^F$-torsor, so
\begin{equation*}
    \left( (q_1)_{|Y^1*}\Lambda \right)^{U_M^F} \simeq \Lambda.
\end{equation*}
Since $(q_1)_{|Y^1}$ is finite, its higher direct images vanish. Taking $U_M^F$-invariants, which commute with derived global sections, gives
\begin{equation}\label{eq:auxiliary-source}
    R\Gamma(Y^1,\Lambda)^{U_M^F} \simeq R\Gamma\left( S_M^1 \times \mathbb G_m, \left( (q_1)_{|Y^1*}\Lambda \right)^{U_M^F} \right) \simeq R\Gamma(S_M^1 \times \mathbb G_m, \Lambda).
\end{equation}

On the boundary, proper base change for the finite morphism $q_1$ gives
\begin{equation}\label{eq:boundary-base-change}
    R(q_1)_{|E^1*}i_1^*Rj_{1,*} \Lambda \simeq i_{\times}^*Rq_{1,*}Rj_{1,*} \Lambda \simeq i_{\times}^*Rj_{\times,*} (q_1)_{|Y^1*} \Lambda.
\end{equation}
Taking $U_M^F$-invariants in \eqref{eq:boundary-base-change} and using \eqref{eq:boundary-invariants} yields 
\begin{equation*}
    R(q_1)_{|E^1*}i_1^*Rj_{1,*} \Lambda \simeq i_{\times}^*Rj_{\times,*}\Lambda.
\end{equation*}
Taking global sections therefore yields
\begin{equation}\label{eq:auxiliary-boundary}
    R\Gamma(E^1, i_1^*Rj_{1,*} \Lambda) \simeq R\Gamma(S_M^1, i_{\times}^*Rj_{\times,*} \Lambda).
\end{equation}
Finally, \eqref{eq:auxiliary-source} and \eqref{eq:auxiliary-boundary} identify the restriction
morphism in \eqref{eq:auxiliary-restriction} with the natural map
\begin{equation}\label{eq:quotient-restriction}
    R\Gamma(S_M^1 \times \mathbb G_m, \Lambda) \to R\Gamma(S_M^1, i_{\times}^*Rj_{\times,*} \Lambda).
\end{equation}
Let $j_{\mathbb G_m}: \mathbb G_m \hookrightarrow \mathbb A^1$ and $i_{\{0\}}: \{0\} \hookrightarrow \mathbb A^1$ be the natural immersions. Since $S_M^1$ is smooth, the Künneth formula and smooth base change identify \eqref{eq:quotient-restriction} with the derived tensor product of $R\Gamma(S_M^1, \Lambda)$ and the map
\begin{equation*}
    R\Gamma(\mathbb G_m, \Lambda) \to i_{\{0\}}^*Rj_{\mathbb G_m,*}\Lambda.
\end{equation*}
The latter morphism is an isomorphism. Indeed, both sides have $\Lambda$ in degree zero, $\Lambda(-1)$ in degree one, and no other non-zero cohomology groups. The morphism identifies the constant classes and, after a Tate twist, sends the Kummer class of $w_i$ to its local Kummer class at zero. This proves that \eqref{eq:auxiliary-restriction} really is an isomorphism. Since it is induced by restriction and since the actions of $N$ and $U_M^F$ commute, this isomorphism is $N$-equivariant.

We now descend the comparison in two steps, following \cite{wangCohomologieCompactificationVarietes2014} Section 5.2.3. First, Lemma \ref{lem:auxiliary-cover-extension} gives finite étale $N$-torsors
\begin{equation*}
    \pi^1: Y^1 \to Y^0, \qquad \overline{\pi}^1_{|E^1}: E^1 \to E^0.
\end{equation*}
Since $\overline{\pi}^1$ itself is étale, formation of $Rj_{0,*}$ commutes with pullback along $\overline{\pi}^1$. Restricting the resulting isomorphism to $E^1$ gives
\begin{equation*}
    (\overline{\pi}^1_{|E^1})^* (i_0^*Rj_{0,*} \Lambda) \simeq i_1^*Rj_{1,*}\Lambda.
\end{equation*}
Write $R\Gamma(N, -)$ for derived $N$-invariants. Cohomological descent for the torsor $\pi^1$ gives
\begin{equation*}
    R\Gamma(Y^0, \Lambda) \simeq R\Gamma\left( N, R\Gamma(Y^1, \Lambda) \right).
\end{equation*}
The actions of $N$ and $U_M^F$ commute, and taking $U_M^F$-invariants is exact. We may therefore commute $U_M^F$-invariants with derived $N$-invariants, obtaining
\begin{equation*}
    R\Gamma(Y^0, \Lambda)^{U_M^F} \simeq R\Gamma\left( N, R\Gamma(Y^1, \Lambda) \right)^{U_M^F} \simeq R\Gamma\left( N, R\Gamma(Y^1, \Lambda)^{U_M^F} \right).
\end{equation*}
On the boundary, we apply cohomological descent to the torsor $\overline{\pi}^1_{|E^1}$ with coefficients in $i_0^*Rj_{0,*}\Lambda$. Using the base-change isomorphism above to identify its pullback to $E^1$, we obtain
\begin{equation*}
    \begin{split}
        R\Gamma(E^0, i_0^*Rj_{0,*}\Lambda) & \simeq R\Gamma \left( N, R\Gamma \left(E^1, (\overline{\pi}^1_{|E^1})^* (i_0^*Rj_{0,*}\Lambda) \right) \right)\\
        & \simeq R\Gamma \left( N, R\Gamma(E^1, i_1^*Rj_{1,*}\Lambda) \right).
    \end{split}
\end{equation*}
These identifications are compatible with restriction. Applying $R\Gamma(N, -)$ to \eqref{eq:auxiliary-restriction} therefore shows that restriction induces an isomorphism
\begin{equation*}
    R\Gamma(Y^0, \Lambda)^{U_M^F} \xrightarrow{\sim} R\Gamma(E^0, i_0^*Rj_{0,*}\Lambda).
\end{equation*}
This morphism is $H$-equivariant.

Second, the decomposition $Y \simeq Y^0 \times^H T^{cF}$ extends by normalization to
\begin{equation*}
    \overline{\pi}^{-1}(C_M^*) \simeq \overline Y^0 \times^H T^{cF}, \qquad E_M \simeq E^0 \times^H T^{cF}.
\end{equation*}
Since $U_M^F$ preserves every connected component and commutes with $T^{cF}$, the restriction morphism for $Y$ is obtained from that for $Y^0$ by induction from $H$ to $T^{cF}$. This induction is exact because $\Lambda[T^{cF}]$ is free over $\Lambda[H]$. We therefore obtain the desired isomorphism
\begin{equation*}
    R\Gamma(Y, \Lambda)^{U_M^F} \xrightarrow{\sim} R\Gamma(E_M, j_M'^*Rj'_*\Lambda).
\end{equation*}

Finally, every maximal proper rational parabolic is $G^F$-conjugate to a standard one. The actions of $G^F$ on $Y$ and $\overline Y$ identify the corresponding restriction morphisms, proving the assertion for every maximal proper rational parabolic $M$. 
\end{proof}

\begin{proof}[Proof of Theorems \ref{thm:wang} and \ref{thm:wang-open}]
Proposition \ref{prop:wang-maximal} applies to every Deligne-Lusztig covering of any Coxeter variety for any pair $(G, F)$. In particular, it applies to $(L_P, F_P)$ for every rational parabolic $P$ of $G$, including $P = G$. The hypothesis of Lemma \ref{lem:wang-reduction} is therefore satisfied, which proves Theorem \ref{thm:wang-open}. The equivalence recalled in Section \ref{subsec:coxeter-cover} then gives Theorem \ref{thm:wang} as well.
\end{proof}

\begin{remark}\label{rem:restriction-adic}
    Taking derived inverse limits shows that Theorems \ref{thm:wang} and \ref{thm:wang-open} also hold with $\mathbb Z_{\ell}$-coefficients, and with $\mathbb Q_{\ell}$-coefficients as well.
\end{remark}

\section{Alvis-Curtis duality}\label{sec:alvis-curtis}

We keep the notations of Section \ref{sec:restriction}. We begin by recalling the complexes defining Alvis-Curtis duality, and then we use Theorem \ref{thm:wang} to compare ordinary and compactly supported cohomology of $Y$.

\subsection{The duality functors}

For a rational parabolic $P$, Harish-Chandra restriction sends a $\Lambda[G^F]$-module $V$ to $V^{U_P^F}$, regarded as a module for $L_P^{F_P} \simeq P^F/U_P^F$. Inflating this module to $P^F$ and inducing to $G^F$ gives $\Lambda[G^F] \otimes_{\Lambda[P^F]} V^{U_P^F}$. Alvis-Curtis duality combines these operations for all the different rational parabolic types. 

Recall that $r = |\Delta/\phi|$ and let $\mathcal P_t$ be the set of rational parabolic subgroups of relative corank $t$. Thus $\mathcal P_0 = \{G\}$ while $\mathcal P_r$ is the set of rational Borel subgroups. For every rational parabolic $P$, we define
\begin{equation*}
    e_P \coloneq \frac{1}{|U_P^F|} \sum_{u \in U_P^F} u \in \Lambda[U_P^F].
\end{equation*}
This idempotent projects $V$ onto $V^{U_P^F} = e_PV$. The direct sum $\bigoplus_{P \in \mathcal P_t} V^{U_P^F}$ carries a $G^F$-action as follows. An element $g \in G^F$ sends $v$ in the summand indexed by $P$ to $g\cdot v$ in the summand indexed by $gPg^{-1}$. Grouping the summands by parabolic type gives a natural identification
\begin{equation*}
    \bigoplus_{P \in \mathcal P_t} V^{U_P^F} \simeq \bigoplus_{\substack{I \subseteq \Delta, \ \phi(I) = I,\\ d(P_I) = t}} \Lambda[G^F] \otimes_{\Lambda[P_I^F]} V^{U_I^F}.
\end{equation*}
Fix an ordering of $\Delta/\phi$. For a bounded complex $V^{\bullet}$ of $\Lambda[G^F]$-modules, we define
\begin{equation}\label{eq:ac-counit}
    D_G(V^{\bullet}) \coloneq \mathrm{Tot} \left[ \bigoplus_{P \in \mathcal P_r} (V^{\bullet})^{U_P^F} \to \cdots \to \bigoplus_{P \in \mathcal P_1}(V^{\bullet})^{U_P^F} \to V^{\bullet} \right],
\end{equation}
where the corank-$t$ column is placed in degree $-t$. For $Q \subset P$ with $d(Q) = d(P)+1$, the corresponding component of the differential is the inclusion $(V^{\bullet})^{U_Q^F} \hookrightarrow (V^{\bullet})^{U_P^F}$, multiplied by the sign determined by the fixed ordering of $\Delta/\phi$. 

We also consider the complex
\begin{equation}\label{eq:ac-unit}
    A_G(V^{\bullet}) \coloneq \mathrm{Tot} \left[ V^{\bullet} \to \bigoplus_{P \in \mathcal P_1}(V^{\bullet})^{U_P^F} \to \cdots \to \bigoplus_{P \in \mathcal P_r} (V^{\bullet})^{U_P^F} \right],
\end{equation}
where the corank-$t$ column is placed in degree $t$. For the same inclusion $Q \subset P$, its differential is $e_Q: (V^{\bullet})^{U_P^F} \to (V^{\bullet})^{U_Q^F}$, multiplied by the sign determined by the fixed ordering of $\Delta/\phi$. 

The functors $D_G$ and $A_G$ are mutually inverse autoequivalences of the derived category $D^b(\Lambda[G^F]\text{-}\mathrm{mod})$, see \cite{cabanes_alvis-curtis_2001} Theorem 5.1 and \cite{cabanes_okuyamas_2009} Theorem 16.

\subsection{Ordinary and compactly supported cohomology}

Theorem \ref{thm:wang} identifies the cohomology of each closed stratum $\overline{C_P}$ with coefficients in $Rj_*\pi_*\Lambda$, with the $U_P^F$-invariants of $R\Gamma(Y, \Lambda)$. We apply these identifications to the \v{C}ech complex for the covering of $\overline X \setminus X$ by its boundary divisors, and we relate it to the complex defining $A_G R\Gamma(Y, \Lambda)$.

\begin{corollary}\label{cor:coxeter-duality}
    There are isomorphisms
    \begin{equation*}
        D_G R\Gamma_c(Y, \Lambda) \simeq R\Gamma(Y, \Lambda), \qquad A_G R\Gamma(Y, \Lambda) \simeq R\Gamma_c(Y, \Lambda)
    \end{equation*}
    in $D^b \left( \Lambda[G^F] \otimes_{\Lambda} \Lambda[T^{cF}]^{\mathrm{op}} \text{-}\mathrm{mod} \right)$. Under the second isomorphism, the projection to the column placed in degree zero in \eqref{eq:ac-unit} corresponds to the natural morphism
    \begin{equation*}
        R\Gamma_c(Y, \Lambda) \to R\Gamma(Y, \Lambda).
    \end{equation*}
\end{corollary}

\begin{proof}
Let $D = (\overline X \setminus X)_{\mathrm{red}}$ with closed immersion $i: D \hookrightarrow \overline X$. Applying $R\Gamma(\overline X, -)$ to the distinguished triangle
\begin{equation*}
    j_!\pi_*\Lambda \to Rj_*\pi_*\Lambda \to i_*i^*Rj_*\pi_*\Lambda \xrightarrow{+1}
\end{equation*}
gives the distinguished triangle
\begin{equation}\label{eq:boundary-localization}
    R\Gamma_c(Y, \Lambda) \to R\Gamma(Y, \Lambda) \to R\Gamma(D, i^*Rj_*\pi_*\Lambda) \xrightarrow{+1}.
\end{equation}
Here, properness of $\overline X$ and finiteness of $\pi$ allow us to identify $R\Gamma(\overline X, j_!\pi_*\Lambda)$ with $R\Gamma_c(Y, \Lambda)$.

We compute $R\Gamma(D, i^*Rj_*\pi_*\Lambda)$ using the finite closed covering of $D$ by the boundary divisors $\overline{C_M}$. If $\overline{C_{M_1}}, \ldots, \overline{C_{M_t}}$ are distinct divisors with nonempty intersection, then
\begin{equation*}
    \overline{C_{M_1}} \cap \cdots \cap \overline{C_{M_t}} = \overline{C_P} \text{ where } P \coloneq M_1 \cap \cdots \cap M_t, \text{ and we have } d(P) = t.
\end{equation*}
The equality follows from \eqref{eq:stratification}, and $d(P) = t$ because the divisors have distinct types. Conversely, every rational parabolic of corank $t$ is the intersection of the $t$ maximal proper rational parabolics containing it.

We order the divisors by their types according to the fixed ordering of $\Delta/\phi$. The restriction maps to the closed strata give the augmented \v{C}ech complex on $\overline X$
\begin{equation*}
    0 \to i_*i^*Rj_*\pi_*\Lambda \to \bigoplus_{P \in \mathcal P_1} (i_P)_*i_P^*Rj_*\pi_*\Lambda \to \cdots \to \bigoplus_{P \in \mathcal P_r} (i_P)_*i_P^*Rj_*\pi_*\Lambda \to 0.
\end{equation*}
The signs of the differentials are determined by the ordering of their types. By \cite{artinTheorieToposCohomologie1972a} Exposé VIII Lemme 8.2 and Remarque 8.3(a), this complex is a resolution of $i_*i^*Rj_*\pi_*\Lambda$.

Applying $R\Gamma(\overline X, -)$ we obtain
\begin{equation*}
    R\Gamma(D, i^*Rj_*\pi_*\Lambda) \simeq \mathrm{Tot} \Bigl[ \bigoplus_{P \in \mathcal P_1} R\Gamma(\overline{C_P}, i_P^*Rj_*\pi_*\Lambda) \to \cdots \to \bigoplus_{P \in \mathcal P_r} R\Gamma(\overline{C_P}, i_P^*Rj_*\pi_*\Lambda) \Bigr],
\end{equation*}
where the corank-$t$ column is placed in degree $t-1$. For $Q \subset P$, we denote the restriction morphism by
\begin{equation*}
    r_{PQ}: R\Gamma(\overline{C_P}, i_P^*Rj_*\pi_*\Lambda) \to R\Gamma(\overline{C_Q}, i_Q^*Rj_*\pi_*\Lambda).
\end{equation*}
The differentials are the maps $r_{PQ}$ with $d(Q) = d(P)+1$, multiplied by adequate sign.

By \eqref{eq:boundary-localization} we have
\begin{equation*}
    R\Gamma_c(Y, \Lambda) \simeq \mathrm{Cone}\left( R\Gamma(Y, \Lambda) \to R\Gamma(D, i^*Rj_*\pi_*\Lambda) \right)[-1].
\end{equation*}
By substitution in the preceding \v{C}ech complex, we obtain
\begin{equation}\label{eq:ac-boundary-complex}
    R\Gamma_c(Y, \Lambda) \simeq \mathrm{Tot}\Bigl[ R\Gamma(Y, \Lambda) \to \bigoplus_{P \in \mathcal P_1} R\Gamma(\overline{C_P}, i_P^*Rj_*\pi_*\Lambda) \to \cdots \to \bigoplus_{P \in \mathcal P_r} R\Gamma(\overline{C_P}, i_P^*Rj_*\pi_*\Lambda) \Bigr],
\end{equation}
with $R\Gamma(Y, \Lambda)$ placed in degree zero and the corank-$t$ column in degree $t$. Projection to the degree zero column is then the natural morphism $R\Gamma_c(Y, \Lambda) \to R\Gamma(Y, \Lambda)$.

We now compare \eqref{eq:ac-boundary-complex} with \eqref{eq:ac-unit} applied to $R\Gamma(Y, \Lambda)$. Using functorial and equivariant resolutions, we choose cochain complexes $V^{\bullet}$ and $V_P^{\bullet}$ representing
\begin{equation*}
    V^{\bullet} \simeq R\Gamma(Y, \Lambda), \qquad V_P^{\bullet} \simeq R\Gamma(\overline{C_P}, i_P^*Rj_*\pi_*\Lambda).
\end{equation*}
They come together with chain maps, where $Q \subset P$, 
\begin{equation*}
    \varepsilon_P: V^{\bullet} \to V_P^{\bullet}, \qquad \widetilde r_{PQ}: V_P^{\bullet} \to V_Q^{\bullet},
\end{equation*}
representing the restriction morphisms and satisfying, for $R \subset Q \subset P$,
\begin{equation*}
    \widetilde r_{PQ} \varepsilon_P = \varepsilon_Q, \qquad \widetilde r_{QR} \widetilde r_{PQ} = \widetilde r_{PR}.
\end{equation*}
By Theorem \ref{thm:wang}, the restriction $(\varepsilon_P)_{|e_PV^{\bullet}}: e_PV^{\bullet} \to V_P^{\bullet}$ is a $P^F$-equivariant quasi-isomorphism. Since $U_P^F$ acts trivially on its source, it also acts trivially on $H^*(V_P^{\bullet})$. Exactness of taking $U_P^F$-invariants therefore shows that the projection $e_P: V_P^{\bullet} \to e_PV_P^{\bullet}$ is a quasi-isomorphism. By composition, we get a quasi-isomorphism
\begin{equation*}
    \varepsilon'_P \coloneq e_P \circ (\varepsilon_P)_{|e_PV^{\bullet}}: e_PV^{\bullet} \to e_PV_P^{\bullet}.
\end{equation*}
We now compare the boundary complex in \eqref{eq:ac-boundary-complex} with the complex defining $A_G(V^{\bullet})$ by mapping both of them to a third complex. All three have $V^{\bullet}$ in horizontal degree zero. The comparison will be given by the following diagram, whose columns are placed from left to right in degree $0, 1, \ldots, r$.
\begin{equation*}
    \begin{tikzcd}[row sep=large]
        V^{\bullet} \arrow[r] \arrow[d, equal] & \bigoplus_{P \in \mathcal P_1} V_P^{\bullet} \arrow[r] \arrow[d, "\bigoplus e_P"] & \cdots \arrow[r] & \bigoplus_{P \in \mathcal P_r}V_P^{\bullet} \arrow[d, "\bigoplus e_P"] \\
        V^{\bullet} \arrow[r] & \bigoplus_{P \in \mathcal P_1}e_PV_P^{\bullet} \arrow[r] & \cdots \arrow[r] & \bigoplus_{P \in \mathcal P_r} e_PV_P^{\bullet} \\
        V^{\bullet} \arrow[r] \arrow[u, equal] & \bigoplus_{P \in \mathcal P_1} e_PV^{\bullet} \arrow[r] \arrow[u, swap, "\bigoplus\varepsilon'_P"] & \cdots \arrow[r] & \bigoplus_{P \in \mathcal P_r}e_PV^{\bullet} \arrow[u, swap, "\bigoplus\varepsilon'_P"]
    \end{tikzcd}
\end{equation*}
The top row represents the boundary complex in \eqref{eq:ac-boundary-complex} and the bottom row defines $A_G(V^{\bullet})$. In the middle row, the first map is $(e_P \varepsilon_P)_{P \in \mathcal P_1}$. For $Q \subset P$ with $d(Q) = d(P)+1$, the corresponding component between two successive terms is $e_Q\widetilde r_{PQ}: e_PV_P^{\bullet} \to e_QV_Q^{\bullet}$, multiplied by the adequate sign. We check that the middle row is a complex and that the diagram commutes.

For $Q \subset P$, the closed stratum $\overline{C_P}$ is $P^F$-stable, while $\overline{C_Q}$ is $Q^F$-stable. Since $U_Q^F \subseteq Q^F \subseteq P^F$, the restriction map $\widetilde r_{PQ}$ is $U_Q^F$-equivariant. Together with $e_Qe_P = e_Q$, this gives
\begin{equation*}
    e_Q \widetilde r_{PQ} e_P = e_Q \widetilde r_{PQ}.
\end{equation*}
Thus the upper squares commute, noting that the upper square involving $V^{\bullet}$ commutes by the definition of the maps $e_P \varepsilon_P$. The projections from the top row to the middle row are surjective in every degree. By diagram chasing, it follows that the middle row is a complex.

Besides, the identities
\begin{equation*}
    \widetilde r_{PQ} \varepsilon_P = \varepsilon_Q, \qquad e_Q\varepsilon_Q = \varepsilon_Qe_Q,
\end{equation*}
give $(e_Q \widetilde r_{PQ}) \varepsilon'_P = \varepsilon'_Qe_Q$. Thus the lower squares commute, noting that the lower square involving $V^{\bullet}$ commutes because $\varepsilon'_Pe_P = e_P\varepsilon_Pe_P = e_P\varepsilon_P$. Hence the diagram does commute.

Every vertical map is a quasi-isomorphism, thus the induced maps between their total complexes are quasi-isomorphisms. The totalization of the top row represents $R\Gamma_c(Y, \Lambda)$ by \eqref{eq:ac-boundary-complex}, while that of the bottom row represents $A_G R\Gamma(Y, \Lambda)$. We therefore obtain the desired identity
\begin{equation*}
    A_G R\Gamma(Y, \Lambda) \simeq R\Gamma_c(Y, \Lambda).
\end{equation*}
Applying the inverse functor $D_G$ gives the other asserted isomorphism. Finally, the restriction maps are compatible with conjugation by $G^F$, and $ge_Pg^{-1} = e_{gPg^{-1}}$ for every $g \in G^F$. Thus both comparisons are $G^F$-equivariant, and they also commute with the right $T^{cF}$-action.
\end{proof}

By Remark \ref{rem:restriction-adic}, the same statements hold with $\mathbb Z_{\ell}$-coefficients or $\mathbb Q_{\ell}$-coefficients.

\bibliographystyle{amsplain}
\bibliography{biblio}

\end{document}